\documentclass[a4paper,11pt,reqno]{amsart}
\usepackage[body={17cm,21cm}]{geometry}
\usepackage[initials]{amsrefs}

\usepackage{amssymb,amscd}
\usepackage[all]{xy}

\usepackage[mathscr]{eucal}
\usepackage{enumerate}

\numberwithin{equation}{section}
\theoremstyle{plain}
\newtheorem{thm}[equation]{Theorem}
\newtheorem{cor}[equation]{Corollary}
\newtheorem{lem}[equation]{Lemma}
\newtheorem{prop}[equation]{Proposition}

\newtheorem{definition}[equation]{Definition}

\newtheorem{rem}[equation]{Remark}

\theoremstyle{definition}

\theoremstyle{remark}

\newtheorem{claim}[equation]{Claim}

\def\Vol{{\rm Vol}}
\def\Hess{{\rm Hess}}
\def\Ric{{\rm Ric}}
\def\Supp{{\rm Supp}}

\DeclareFontFamily{U}{wncy}{}
\DeclareFontShape{U}{wncy}{m}{n}{%
<5>wncyr5%
<6>wncyr6%
<7>wncyr7%
<8>wncyr8%
<9>wncyr9%
<10>wncyr10%
<11>wncyr10%
<12>wncyr6%
<14>wncyr7%
<17>wncyr8%
<20>wncyr10%
<25>wncyr10}{}
\DeclareMathAlphabet{\cyr}{U}{wncy}{m}{n}
\begin{document}

\title[The continuity equation with cusp singularities]
{The continuity equation with cusp singularities}

\author{Yan Li}

\address{Yan Li \newline Beijing International Center for Mathematical Research, \newline Peking University,
\newline 100871 Beijing, China}
\email{liyandota@hotmail.com}

\date{\today}

\maketitle

\begin{abstract}
In this paper we study a special case of the completion of cusp K\"{a}hler-Einstein metric on the regular part of varieties by taking the continuity method proposed by La Nave and Tian. The differential geometric and algebro-geometric properties of the noncollapsing limit in the continuity method with cusp singularities will be investigated.

\end{abstract}

\section{Introduction}
The Yau-Tian-Donaldson conjecture for Fano manifolds has revealed deep connections among complex Monge-Amp\`{e}re equation, metric geometry and complex algebraic geometry. Some specialists develop many techniques to deal with the celebrated conjecture, see \cite{Ti15} or \cite{CDS1}, \cite{CDS2}, \cite{CDS3}. These methods also play an important role in studying many other problems. For instance, in \cite{So}, J. Song proves that the metric completion of the regular set of Calabi-Yau varieties and canonical models of general type with crepant singularities is a compact length space which homeomorphic to the original variety. In \cite{NT}, G. La Nave and G. Tian introduce a new continuity equation to consider the analytic minimal model program. Later in \cite{NTZ}, G. La Nave, G. Tian and Z. L. Zhang study the differential geometric and algebro-geometric properties of the noncollapsing limit in the continuity equation. These fundamental results focus on the compact K\"{a}hler manifolds. For noncompact case, let us recall some facts. Suppose $\overline{M}$ is a compact complex manifold, $D$ is an effective divisor with only normal crossings and $K_{\overline{M}}+D$ is ample, where $K_{\overline{M}}$ is the canonical line bundle over $\overline{M}$. A well known result achieved by Kobayashi \cite{Ko84} and Tian-Yau \cite{TY87} asserts that there exists a complete negative K\"{a}hler-Einstein metric on $\overline{M}\setminus D$. Recently, in \cite{BG}, two authors generalize this result. For the convenience, this consequence is stated as following (Theorem C \cite{BG}).
\begin{thm}
Let $\overline{M}$ be a compact K\"{a}hler manifold and $D$ is a simple normal crossings $\mathbb{R}$-divisor on $\overline{M}$ with coefficients in $[-1, +\infty)$ such that $K_{\overline{M}}+D$ is semi-positive and big. Then there exists a unique $\omega$ in $c_1(K_{\overline{M}}+D)$ which is smooth on a Zariski open set $U$ of $\overline{M}$ and such that
$$
\Ric(\omega)=-\omega+[D].
$$
More precisely, $U$ can be taken to the $\overline{M}\setminus (D\cup \mathcal{S})$, where $\mathcal{S}$ is the intersetion of all effective $\mathbb{Q}$-divisors $E$ such that $K_{\overline{M}}+D-E$ is ample.
\end{thm}
Motivated by \cite{So}, a natural problem is to ask what the completion of $(U, \omega)$ is. In this article, a special case is investigated. More precisely, suppose $\overline{M}$ is a projective manifold, $D$ is a smooth hypersurface and $K_{\overline{M}}+D$ is big and semi-ample. According to the Kawamata base point free theorem, there exists an integer $K\in \mathbb{Z}^+$ such that an orthonormal basis of $K(K_{\overline{M}}+D)$ gives a holomorphic map
$$
\Phi: \overline{M}\rightarrow \mathbb{CP}^N
$$
where $N=\dim H^0(\overline{M}, K(K_{\overline{M}}+D))-1$. The mainly result in this article is that the completion of $(U, \omega)$ in the sense of Theorem 1 homeomorphic to $\Phi(\overline{M}\setminus D)$. If the divisor $D$ is simple normal crossing with coefficients $1$, then there is a similar result. To dealing with this problem, the continuity method is taken proposed by G. La Nave and G. Tian in \cite{NT}.

To begin with, let $\overline{M}$ be a projective manifold with a K\"{a}hler metric $\omega_0$ and $D$ be a smooth hypersurface in $\overline{M}$ such that $K_{\overline{M}}+D$ is big and semi-ample. $h_D$ is denoted by the hermitian metric on $L_D$, the associated line bundle of $D$ such that $\omega_0-\sqrt{-1}\partial\overline{\partial}\log\log^2|s_D|^2_{h_D}>0$, where $s_D$ is the defining section of $D$. The following $1$-parameter family equations are considered:
\begin{equation}\label{e1}
(1+t)\omega=\omega_0-\sqrt{-1}\partial\overline{\partial}\log\log^2|s_D|^2_{h_D}-t(\Ric(\omega)-[D]),
\end{equation}
where $[D]$ is the current of integratioin along $D$.

Recall that $\omega$ is sald to have cusp singularities along $D$ if, whenever $D$ is locally given by $(z_1=0)$, $\omega$ is quasi-isometric to the cusp metric
$$
\omega_{cusp}=\frac{\sqrt{-1}dz_1\wedge d\bar{z_1}}{|z_1|^2\log^2|z_1|^2}+\sum_{k=2}^ndz_k\wedge d\bar{z_k}.
$$
Since $\omega_0-\sqrt{-1}\partial\overline{\partial}\log\log^2|s_D|^2_{h_D}$ is a K\"{a}hler metric on $\overline{M}\setminus D$ with cusp singularities, the equation (\ref{e1}) essentially state the variation of cusp K\"{a}hler metric along $t$. Therefore, the equation (\ref{e1}) is called the cusp continuity equation.
\begin{thm}\label{t1}
The cusp continuity equation (\ref{e1}) is solvable for all $t\in [0,+\infty)$.
\end{thm}
$\omega_t$ is denoted by the solution of (\ref{e1}), then we have the following convergence result.
\begin{thm}\label{t2}
$\omega_t$ converge to a unique weakly K\"{a}hler metric $\omega_1$ such that $\omega_1$ is smooth on $\overline{M}\setminus (D\cup \mathcal{S})$ and satisfies
$$
\Ric(\omega_1)=-\omega_1, \ on \ \overline{M}\backslash (\mathcal{S}_{\overline{M}}\cup D),
$$
where
$$
\mathcal{S}_{\overline{M}}=\bigcap\{E|E \ is \ an \ effective \ disivior \ such \ that \ K_{\overline{M}}+D-\rho E>0 \ for \ some \ \rho>0  \}.
$$
\end{thm}
If $G$ is a big divisor, we denoted $B_+(G)$ by the intersection of all effective $\mathbb{Q}$-divisors $E$ such that $G-E$ is ample. Then $\mathcal{S}_{\overline{M}}$ appeared in Theorem \ref{t2} is $B_+(K_{\overline{M}}+D)$. Observing that $\Phi: (\overline{M},D)\rightarrow (\Phi(\overline{M}),\Phi(D))$ can be viewed as a resolution of $(\Phi(\overline{M}),\Phi(D))$ and $K_{\overline{M}}+D=\Phi^*(K_{\Phi(\overline{M})}+\Phi(D))$. According to Theorem 0-3-12 \cite{Kw}, $K_{\Phi(\overline{M})}+\Phi(D)$ is ample. Let $\Phi(\overline{M})_{reg}$ be the regular part of $\Phi(\overline{M})$ and $\Phi(\overline{M})_{sing}$ be the singular part of $\Phi(\overline{M})$. The following Proposition illustrate the connection between $B_+(K_{\overline{M}}+D)$ and $\Phi(\overline{M})_{sing}$, due to Proposition 2.3 \cite{Bo}.
\begin{prop}\label{p1}
Let $\pi: X\rightarrow Y$ be a birational morphism between normal projective varieties. For any big $\mathbb{R}$-divisor $G$ on $Y$ and any effective $\pi$-exceptional divisor $\mathbb{R}$-divisor $F$ on $X$, then we have
$$
B_+(\pi^*G+F)=\pi^{-1}(B_+(G))\cup Exc(\pi),
$$
where $Exc(\pi)\subset X$ is the set of points $x\in X$ such that $\pi$ is not biregular.
\end{prop}
From Theorem \ref{t2} and Proposition \ref{p1}, the following Corollary is derived immediately.
\begin{cor}
$$
\mathcal{S}_{\overline{M}}=B_+(K_{\overline{M}}+D)=\Phi^{-1}(B_+(K_{\Phi(\overline{M})}+\Phi(D)))\cup Exc(\Phi)=\Phi^{-1}(\Phi(\overline{M})_{sing})
$$
and
$$
\overline{M}\backslash (\mathcal{S}_{\overline{M}}\cup D)=\overline{M}_{reg}\setminus D,
$$
where $\overline{M}_{reg}$ represents $\Phi^{-1}(\Phi(\overline{M})_{reg})$. Furthermore the metric $\omega_1$ is smooth on $\overline{M}_{reg}\setminus D$.
\end{cor}
\begin{rem}
If the codimension of $\Phi(D)$ is not $1$, then $D$ is an exceptional divisor of the resolution $\Phi$. Thus, $\overline{M}\backslash (\mathcal{S}_{\overline{M}}\cup D)=\overline{M}_{reg}$.
\end{rem}

$M$ is denoted by $\overline{M}\setminus D$. The next result states that the limit space $(M, \omega_t)$ converge to in the Gromov-Hausdorff topology has more regular properties, such as metric structure, algebraic structure.
\begin{thm}\label{pr}
The following results are hold.
\begin{enumerate}
\item $(M, \omega_t, x)$ converges in the Gromov-Hausdorff topology to a length space $(M_1, d_1, x_1)$ which is the metric completion $(\overline{M}_{reg}\setminus D, \omega_1)$.
\item $M_1=\mathcal{R}\cup \mathcal{S}$ and $\mathcal{R}=\overline{M}_{reg}\setminus D$, where $\mathcal{R}$ is the regular part and $\mathcal{S}$ is the singular part.
\item $\mathcal{R}$ is geodesically convex and $\mathcal{S}$ is closed set which has codimension $\geq 2$.
\item $M_1$ homeomorphic to a normal quasi-subvariety $\Phi(\overline{M}\setminus D)$.
\end{enumerate}
\end{thm}

\textbf{Acknowledgement:}The author would like to thank Professor Gang Tian for his constant help, support and encouragement.

\section{Preliminaries}
In this section, we list some fundamental definitions and results which will be used in the later.
\begin{definition}
Let $V$ be an open set in $\mathbb{C}^n$. A holomorphic map from $V$ into a complex manifold $M$ of complex dimension $n$ is called a quasi-coordinate map if it is of maximal rank everywhere in $V$. This open set $V$ is called a local quasi-coordinate of $M$.
\end{definition}
\begin{definition}
Let $M$ be a complete K\"{a}hler manifold and $\omega$ is the K\"{a}hler form. $(M, \omega)$ is called bounded geometry if there is a quasi-coordinates $\Gamma=\{(V;v^1,\cdot\cdot\cdot,v^n)\}$ which satisfies the following three conditions:
\begin{enumerate}
\item $M$ is covered by the image of $(V;v^1,\cdot\cdot\cdot,v^n)$.
\item The complement of some open neighborhood of $D$ is covered by a finite of $(V;v^1,\cdot\cdot\cdot,v^n)$ which are local coordinates in the usual sense.
\item There exist positive constants $C$ and $A_k$ $(k=0,1,2,\cdot\cdot\cdot)$ independent of $\Gamma$ such that at each $(V;v^1,\cdot\cdot\cdot,v^n)$, the inequalities
$$
 \frac{1}{C} \delta_{ij}<(g_{i\bar{j}})<C\delta_{ij},
$$
$$
\Big|\frac{\partial^{|p|+|q|}}{\partial v^{|p|}\partial \bar{v}^{|q|}}g_{i\bar{j}}\Big|<A_{|p|+|q|}, \ for \ any \ multiindices \ p \ and \ q \,
$$
hold, where $g_{i\bar{j}}$ denote the component of $\omega$ with respect to $V$.
\end{enumerate}
\end{definition}
Now we define the H\"{o}lder space of $C^{k,\lambda}$-functions on a complete K\"{a}hler manifold $(M, \omega)$ which cover by the image of quasi-coordinates. For a nonnegative integer $k$ , $\lambda\in (0,1)$ and $u\in C^k(M)$, we define
\begin{align*}
& ||u||_{k,\lambda}=\sup_{V\in \Gamma}\Big\{\sup_{z\in V}\Big(\sum_{|p|+|q|\leq k}\Big|\frac{\partial^{|p|+|q|}}{\partial v^{|p|}\partial \bar{v}^{|q|}}u(z)\Big|\Big) \\
& +\sup_{z,z'\in V}\Big(\sum_{|p|+|q|=k}|z-z'|^{-\lambda}\Big|\frac{\partial^{|p|+|q|}}{\partial v^{|p|}\partial \bar{v}^{|q|}}u(z)-\frac{\partial^{|p|+|q|}}{\partial v^{|p|}\partial \bar{v}^{|q|}}u(z')\Big|\Big)\Big\}.
\end{align*}
The function space $C^{k,\lambda}(M)$ is, by definition,
$$
C^{k,\lambda}(M)=\{u\in C^{k}(M); ||u||_{k,\lambda}<\infty\},
$$
which is a Banach space with respect to the norm $||\cdot||_{k,\lambda}$.

Next we state the generalized maximum principle, due to Yau (Proposition 1.6 \cite{CY}).
\begin{thm}
Suppose $(M,\omega)$ is a complete K\"{a}hler manifold with bounded geometry and $f$ is a function on $M$ which is bounded from above. Then there exists a sequence ${x_i}$ in $M$ such that $\lim_{i\rightarrow \infty}f(x_i)=\sup f$, $\lim_{i\rightarrow \infty}|\nabla f(x_i)|=0$ and $\overline{\lim}_{i\rightarrow \infty}\Hess(f)(x_i)\leq 0$, where the Hessian is taken with respect to $\omega$.
\end{thm}
Now we introduce the Bochner formula on a general line bundle. Let $(M,\omega)$ be a K\"{a}hler manifold of dimension $n$ and $(L,h)$ be a Hermitian Line bundle over $M$. Let $\Theta_{h}$ be the Chern curvature form of $h$. Let $\nabla$ and $\overline{\nabla}$ denote the $(1, 0)$ and $(0, 1)$ part of
a connection respectively. The connection appeared in this paper is usually known as the Chern connection or Levi-Civita connection.

For a holomorphic section $\tau \in H^0(M,L)$ we write for simplicity
$$
 |\tau|=|\tau|_{h}, \ |\nabla \tau|_{h\otimes \omega}=|\nabla \tau|,
$$
and
$$
|\nabla\nabla \tau|^2=\sum_{i,j}|\nabla_{i}\nabla_{j}\tau|^2, \ |\overline{\nabla}\nabla \tau|^2=\sum_{i,j}|\nabla_{\bar{i}}\nabla_{j}\tau|^2.
$$
By direct computation we have
\begin{lem}(Bochner formulas).\label{Bo}
For any $\tau \in H^0(M,L)$ one has
\begin{equation}\label{e4.2}
 \triangle_{\omega}|\tau|^2=|\nabla \tau|^2-|\tau|^2\cdot tr_{\omega}\Theta_{h}
\end{equation}
and
\begin{multline}
 \triangle_{\omega}|\nabla \tau|^2=|\nabla\nabla \tau|^2+|\overline{\nabla}\nabla\tau|^2   -\nabla_j(\Theta_h)_{i\bar{j}} \langle \tau,\nabla_{\bar{i}}\bar{\tau}\rangle-\nabla_{\bar{j}}(tr_{\omega}\Theta_h)\langle \nabla_j \tau,\bar{\tau}\rangle  \\
 +R_{i\bar{j}}\langle \nabla_j \tau,\nabla_{\bar{i}}\bar{\tau}\rangle
 -2(\Theta_h)_{i\bar{j}}\langle\nabla_j\tau,\nabla_{\bar{i}}\bar{\tau}\rangle-|\nabla \tau|^2\cdot tr_{\omega}\Theta_h
\end{multline}
where $R_{i\bar{j}}$ is the Ricci curvature of $\omega$, $\langle,\rangle$ is the inner product defined by $h$.
\end{lem}
\section{Existence and uniqueness of cusp continuity equation}
This section is devoted to prove the Theorem \ref{t1}. When $t=0$ the equation (\ref{e1}) has a solution $\omega(0)=\omega_0-\sqrt{-1}\partial\bar{\partial}\log\log^2|s_D|^2_{h_D}$. For a fixed $t\neq 0$, we reduce (\ref{e1}) to a scalar equation. First, the background metric will be constructed. Since $K_{\overline{M}}+D$ is semi-ample, the Kawamata base point free Theorem claims that there exists an integer $K_0$ such that $K_0(K_{\overline{M}}+D)$ has no base point. Then a basis of $H^0(\overline{M},K_0(K_{\overline{M}}+D))$ gives a holomorphic map
$$
\Phi: \overline{M}\rightarrow \mathbb{CP}^N
$$
where $N=\dim H^0(\overline{M}, K_0(K_{\overline{M}}+D))-1$. $\omega_{FS}$ is denoted by the Fubini-Study metric on $\mathbb{CP}^N$. Set $\eta_1=\frac{1}{K_0}\Phi^*(\omega_{FS})$. Since $\eta_1\in c_1(K_{\overline{M}}+D)$, there exist a smooth volume form $\Omega$ on $\overline{M}$ and hermitian metric $h_D^t$ on $L_D$ such that $\eta_1=-\Ric(\Omega)+\Theta_{h_D^t}$ and $\frac{1}{1+t'}\omega_0+\frac{t'}{1+t'}\eta_1-\sqrt{-1}\partial\overline{\partial}\log\log^2|s_D|_{h_D^t}>0$ for $t'\in[0,t]$, where $\Theta_{h_D^t}$ is the curvature form of $L_D$ with the metric $h_D^t$. Set $l=\frac{t}{1+t}$, then $\widetilde{\omega_l}:=(1-l)\omega_0+l\eta_1-\sqrt{-1}\partial\overline{\partial}\log\log^2|s_D|_{h_D^t}$ is chosen as the background meric. Therefore the equation (\ref{e1}) is reduced to the following scalar equation
$$
(\widetilde{\omega_l}+\sqrt{-1}\partial\overline{\partial}u_l)^n=e^{\frac{u_l}{l}}\frac{\Omega_l}{|s_D|^2_{h_D^l}\log^2|s_D|^2_{h_D^l}},
$$
where $|s_D|^2_{h_D^l}$ is denoted by $|s_D|^2_{h_D^t}$ and $\Omega_l=\Omega\Big(\frac{\log^2|s_D|^2_{h_D}}{\log^2|s_D|^2_{h_D^l}}\Big)^{\frac{1-l}{l}}$ is a smooth volume form on $\overline{M}$. For the convenience, we simplified the notation of the above equation as following
\begin{equation}\label{e3.1}
(\widetilde{\omega_l}+\sqrt{-1}\partial\overline{\partial}u_l)^n=e^{\frac{u_l}{l}}\frac{\Omega_l}{|s_D|^2\log^2|s_D|^2}.
\end{equation}
To get a complete metric, we define an open subset $U$ in $C^{k,\lambda}(M)$ by
$$
U=\{v\in C^{k,\lambda}(M)| \frac{1}{C}\widetilde{\omega_l}<\widetilde{\omega_l}+\sqrt{-1}\partial\overline{\partial}v<C\widetilde{\omega_l}, \ for \ \ some \ \ positive\ \ constant\ \ C\ \},
$$
where $M=\overline{M}\setminus D$. If $u_l$ belongs to $U$ and satisfies (\ref{e3.1}), then $\widetilde{\omega_l}+\sqrt{-1}\partial\overline{\partial}u_l$ is a complete K\"{a}hler metric.

Now we take the continuity method to solve the equation (\ref{e3.1}). Consider the following equations
\begin{equation}\label{Ce}
(\widetilde{\omega_l}+\sqrt{-1}\partial\overline{\partial}u_{l,s})^n=e^{\frac{u_{l,s}}{l}}\cdot e^{sF}\cdot \widetilde{\omega_l}^n,
\end{equation}
where $F=\frac{\Omega_l}{\widetilde{\omega_l}^n|s_D|^2\log^2|s_D|^2}$. We consider

We consider the $C^0$ map $\Psi: C^{k,\lambda}(M)\rightarrow C^{k-2,\lambda}(M)$ defined by $\Psi(v)=e^{-\frac{u_l}{l}}\cdot\frac{(\widetilde{\omega_l}+\sqrt{-1}\partial\overline{\partial}v)^n}{\widetilde{\omega_l}^n}$. Define
$$
S=\{s\in[0,1]|\ there \ is \ a \ solution\ u_{l,s} \ satisfies \ \Psi(u_{l,s})=e^{sF}\}.
$$
Obviously, $0\in S$. To prove $1\in S$, it is sufficient to show that $S$ is open and closed. The inverse mapping theorem implies the openness. The Fr\'{e}chet derivative $\Psi'(u_{l,s}):  C^{k,\lambda}(M)\rightarrow C^{k-2,\lambda}(M)$ at $u_{l,s}\in U$ is given by
$$
h\rightarrow e^{sF}(\triangle_{\omega_{l,s}}h-\frac{h}{l}),
$$
where $\omega_{l,s}=\widetilde{\omega_l}+\sqrt{-1}\partial\overline{\partial}u_{l,s}$. Due to Kobayashi \cite{Ko84}, $F\in C^{k-2,\lambda}(M)$. Therefore, we have to show that, for any $w\in C^{k-2,\lambda}(M)$,
\begin{equation}\label{e3.2}
\triangle_{\omega_{l,s}}h-\frac{h}{l}=w
\end{equation}
can be solved for $h\in C^{k,\lambda}(M)$ and that $|h|_{C^{k,\lambda}}\leq C|w|_{C^{k-2,\lambda}}$ for some constant $C$ independent of $w$.

We first to show that there is at most one function $h$ in $C^{k,\lambda}(M)$ solving the equation (\ref{e3.2}). It suffices to verify that $\triangle_{\omega_{l,s}}h-\frac{h}{l}=0$ and $h\in C^{k,\lambda}(M)$ imply $h\equiv 0$. Note that $\omega_{l,s}$ is complete K\"{a}hler metric with bounded geometry due to Lemma 2 \cite{Ko84} and Proposition 1.4 \cite{CY}. For such a metric we can use the generalized maximum principle. Suppose $h\in C^{k,\lambda}(M)$, $h$ is in particular bounded. The generalized maximum principle implies that there exists a sequence of points $\{x_i\}$ in $M$ such that $\lim_{i\rightarrow \infty}h(x_i)=\sup h$ and $\overline{\lim}_{i\rightarrow \infty}\triangle_{\omega_{l,s}}h(x_i)\leq 0$. We immediately see that $\sup h\leq 0$ according to the equation $\triangle_{\omega_{l,s}}h-\frac{h}{l}=0$. Similarly, $\inf h\geq0$ and $h\equiv0$.

Now we prove the existence of $h$. Let $\{\Omega_i\}$ be an exhaustion of $M$ by compact subdomains. Suppose $w\in C^{k-2,\lambda}(M)$ and let $h_i$ be the unique solution to
$$
\triangle_{\omega_{l,s}}h_i-\frac{h_i}{l}=w  \ in \ \Omega_i,
$$
$$
h_i=0 \ on \ \partial\Omega_i.
$$
The maximum principle applied to $\Omega_i$ shows that
$$
\sup_{\Omega_i}|h_i|\leq l\cdot \sup_{\Omega_i}|w|.
$$
Interior Schauder estimates shows that a sequence of $h_i$ converge to some $h\in C^{k,\lambda}(M)$ which solves the equation (\ref{e3.2}) and that the estimate $|h|_{C^{k,\lambda}}\leq C|w|_{C^{k-2,\lambda}}$.

Next, it remains to show that $S$ is closed. Assume that $\{s_i\}\subset E$ is a sequence with $\lim_{i\rightarrow\infty}s_i=\bar{s}$ and $u_{l,s_i}$ is the solution of (\ref{Ce}) with $s=s_i$. We want to prove $\bar{s}\in E$. It amounts to getting a prior $C^{k,\lambda}(M)$-estimate for each $u_{l,s_i}$. By applying the generalized maximum principle to (\ref{Ce}), we have
$$
\sup_{M}|u_{l,s_i}|\leq l\cdot s_i \sup_{M}|F|\leq C \sup_{M}|F|.
$$
So we have the $C^0$-estimate due to Lemma 1 \cite{Ko84}. For the $C^2$-estimate, since $(M,\widetilde{\omega_l})$ is a bounded geometry, by the standard calculation we have
$$
\Ric(\omega_{l,s_i})=-\frac{1}{l}\omega_{l,s_i}+\frac{1}{l}\widetilde{\omega_l}-s_i\sqrt{-1}\partial\overline{\partial}F+\Ric(\widetilde{\omega_l})\leq -\frac{1}{l}\omega_{l,s_i}+C\widetilde{\omega_l}
$$
and
$$
\triangle_{\omega_{l,s_i}}\log tr_{\widetilde{\omega_l}}\omega_{l,s_i}\geq \frac{1}{tr_{\widetilde{\omega_l}}\omega_{l,s_i}}\big(-g^{i\bar{j}}(\widetilde{\omega_l})R_{i\bar{j}}(\omega_{l,s_i})+g^{i\bar{j}}(\omega_{l,s_i})g_{k\bar{l}}(\omega_{l,s_i}){R_{i\bar{j}}}^{k\bar{l}}(\widetilde{\omega_l})\big).
$$
Then
$$
\triangle_{\omega_{l,s_i}}\log tr_{\widetilde{\omega_l}}\omega_{l,s_i}\geq -atr_{\omega_{l,s_i}}\widetilde{\omega_l}-\frac{A}{tr_{\widetilde{\omega_l}}\omega_{l,s_i}}-C,
$$
where $-a$ is the lower bound of holomorphic bisectional curvature of metric $\widetilde{\omega_l}$. Note that
$$
\triangle_{\omega_{l,s_i}}u_{l,s_i}=n-tr_{\omega_{l,s_i}}\widetilde{\omega_l}.
$$
Let $H=\log tr_{\widetilde{\omega_l}}\omega_{l,s_i}-(a+1)u_{l,s_i}$, then
$$
\triangle_{\omega_{l,s_i}}H\geq tr_{\omega_{l,s_i}}\widetilde{\omega_l}-\frac{A}{tr_{\widetilde{\omega_l}}\omega_{l,s_i}}-C.
$$
By the generalized maximum principle, there exists a sequence $\{x_i\}$ such that $\lim_{i\rightarrow\infty}H(x_i)=\sup H$ and $\overline{\lim}_{i\rightarrow\infty}\sqrt{-1}\partial\overline{\partial}H(x_i)\leq 0$. So we have a subsequence also denoted by $\{x_i\}$ such that
$$
\lim_{i\rightarrow\infty}tr_{\widetilde{\omega_l}}\omega_{l,s_i}(tr_{\omega_{l,s_i}}\widetilde{\omega_l}-C)(x_i)\leq A.
$$
Nota that $(tr_{\widetilde{\omega_l}}\omega_{l,s_i})^{\frac{1}{n-1}}\leq C'tr_{\omega_{l,s_i}}\widetilde{\omega_l}$. Then we get
\begin{equation}\label{e4.0}
\lim_{i\rightarrow\infty}tr_{\widetilde{\omega_l}}\omega_{l,s_i}\Big(\frac{1}{C'}(tr_{\widetilde{\omega_l}}\omega_{l,s_i})^{\frac{1}{n-1}}-C\Big)(x_i)\leq A.
\end{equation}
If $$\lim_{i\rightarrow\infty}(tr_{\widetilde{\omega_l}}\omega_{l,s_i})^{\frac{1}{n-1}}(x_i)\leq 2C'C,$$ then we see $$\lim_{i\rightarrow\infty}tr_{\widetilde{\omega_l}}\omega_{l,s_i}(x_i)\leq C.$$
Otherwise $$\lim_{i\rightarrow\infty}(tr_{\widetilde{\omega_l}}\omega_{l,s_i})^{\frac{1}{n-1}}(x_i)\geq 2C'C,$$ then by (\ref{e4.0}) we have $$\lim_{i\rightarrow\infty}tr_{\widetilde{\omega_l}}\omega_{l,s_i}(x_i)\leq C.$$
Therefore, $H\leq C$. This implies $tr_{\widetilde{\omega_l}}\omega_{l,s_i}\leq C$. Furthermore by a standard inequality, we get $C^{-1}\widetilde{\omega_l}\leq \omega_{l,s_i}\leq C\widetilde{\omega_l}$.

For the $3$-order estimate, let $T=|\nabla_{\widetilde{g}_l}\partial\overline{\partial}u_{l,s_i}|^2_{g_{l,s_i}}$, where $\widetilde{g}_l$ and $g_{l,s_i}$ represent Riemannian metrics associated with K\"{a}hler forms $\widetilde{\omega}_l$ and $\omega_{l,s_i}$. By a standard computations (c.f. Proposition 4.3 \cite{CY}), we have
$$
\triangle_{\omega_{l,s_i}}(T+C\triangle_{\widetilde{\omega_l}} u_{l,s_i})\geq C_1 T-C_2.
$$
By the Laplace estimate of $u_{l,s_i}$ and generalized maximum principle, we get $T\leq C$. Thus, by taking a subsequence if necessary, $u_{l,s_i}$ $C^{2,\lambda}$-converge to a solution with $s=\bar{s}$. This implies $S$ is closed.

Next we prove the uniqueness of equation (\ref{e3.1}). Suppose that $u_{l,1}$ and $u_{l,2}$ are solutions to (\ref{e3.1}). Set $\omega_2=\widetilde{\omega_l}+\sqrt{-1}\partial\overline{\partial}u_{l,2}$, then we have
$$
\frac{(\omega_2+\sqrt{-1}\partial\overline{\partial}(u_{l,1}-u_{l,2}))^n}{\omega_2^n}=e^{\frac{u_{l,1}-u_{l,2}}{l}}.
$$
Since $(M,\omega_2)$ is a complete K\"{a}hler manifold with bounded geometry (c.f. Proposition 1.4 \cite{CY}), applying the generalized maximum principle, there exists a sequence $\{x_i\}$ such that $\lim_{i\rightarrow\infty}(u_{l,1}-u_{l,2})(x_i)=\sup_{M}(u_{l,1}-u_{l,2})$ and $\overline{\lim}_{i\rightarrow\infty} \Hess(u_{l,1}-u_{l,2})(x_i)\leq 0$. Furthermore, we obtain $u_{l,1}\leq u_{l,2}$. By the same argument, we have $u_{l,1}\geq u_{l,2}$. Therefore, the equation (\ref{e3.1}) has only one solution. Finally, the cusp continuity equation is solvable for all $t\in [0,\infty)$ i.e., $l\in[0,1)$.

\section{Convergence of cusp continuity equation}
In this section we investigate the regular properties of limit metric.
\begin{lem}\label{L1}
Let $F$ be a divisor on $\overline{M}$. If $F$ is nef and big, then there is an effective divisor $E=\sum_ia_iE_i$ such that $F-\epsilon E>0$ for all sufficiently small $\epsilon>0$.
\end{lem}
By the assumption that $K_{\overline{M}}+D$ is big and semi-ample, there exists an effective divisor $E=\sum_ia_iE_i$ such that $K_{\overline{M}}+D-\epsilon E>0$ for all sufficiently small $\epsilon>0$ according to Lemma \ref{L1}. Thus we choose a volume form $\Omega$, a hermitian metric $h'_D$ on $L_D$ and hermitian metrics $h_{E_i}$ such that
$$
-\Ric(\Omega)+\Theta'_D-\sum_i\epsilon a_i\Theta_{E_i}>0,
$$
where $\Theta'_D$ and $\Theta_{E_i}$ represent curvature forms of line bundles $L_D$ and $L_{E_i}$ associated with metrics $h'_D$ and $h_{E_i}$ respectively. $s_D$ and $s_{E_i}$ are denoted by the defining sections of $L_D$ and $L_{E_i}$. For simplicity, we write $\log|s_{E}|^2=\sum_ia_i\log|s_{E_i}|^2$. By taking appropriate $\Omega$, $h'_D$ and $h_{E_i}$, we can assume that
$$
-\Ric(\Omega)+\Theta'_D-\sum_i\epsilon a_i\Theta_{E_i}>0,
$$
and
$$
\widetilde{\omega_{l,E}}:=(1-l)\omega_0+l(-\Ric(\Omega)-\sqrt{-1}\partial\overline{\partial}\log|s_D|^2_{h'_D}+\epsilon\sqrt{-1}\partial\overline{\partial}\log|s_{E}|^2)-\sqrt{-1}\partial\overline{\partial}\log\log^2|s_D|^2_{h'_D}>0
$$
for $l\in[\frac{1}{2},1]$.
Let $\widetilde{\omega_l}:=(1-l)\omega_0+l(-\Ric(\Omega)+\Theta_D)-\sqrt{-1}\partial\overline{\partial}\log\log^2|s_D|^2_{h_D}$ (may not be a metric), where the hermitian metric $h_D$ is defined as $\omega(0)=\omega_0-\sqrt{-1}\partial\overline{\partial}\log\log^2|s_D|^2_{h_D}>0$ and $\Theta_D$ is the curvature form of $h_D$. Then the equation (\ref{e1}) is written as
\begin{equation}\label{e51}
(\widetilde{\omega_l}+\sqrt{-1}\partial\overline{\partial}u_l)^n=e^{\frac{u_l}{l}}\cdot \frac{\Omega}{|s_D|^2_{h_D}\log^2|s_D|^2_{h_D}}.
\end{equation}
This equation is also equivalent to
\begin{equation}\label{e52}
(\widetilde{\omega_{l,E}}+\sqrt{-1}\partial\overline{\partial}w_l)^n=e^{\frac{1}{l}w_l+\epsilon\log|s_E|^2}\cdot \frac{\Omega_l}{|s_D|^2_{h'_D}\log^2|s_D|^2_{h'_D}},
\end{equation}
where $w_l=u_l-l\epsilon\log|s_E|^2+l\log\frac{|s_D|^2_{h'_D}}{|s_D|^2_{h_D}}+\log\frac{\log^2|s_D|^2_{h'_D}}{\log^2|s_D|^2_{h_D}}$ and $\Omega_l=\Omega\Big(\frac{\log^2|s_D|^2_{h_D}}{\log^2|s_D|^2_{h'_D}}\Big)^{\frac{1-l}{l}}$.
\begin{lem}\label{L4.1}
There exists a constant $C$ independent of $l$ such that $-C\leq w_l\leq C-l\epsilon \log|s_E|^2$.
\end{lem}
\begin{proof}
For the lower bound, we note that $\widetilde{\omega_{l,E}}=(1-l)\omega_0+l(-\Ric(\Omega)+\Theta'_D-\sum_i\epsilon a_i\Theta_{E_i})-\sqrt{-1}\partial\overline{\partial}\log\log^2|s_D|^2_{h'_D}$ is a complete K\"{a}hler metric with bounded geometry on $M$. Applying the generalized maximum principle to (\ref{e52}), we get $w_l\geq -C-l\epsilon\log|s_E|^2\geq -C$.

For the upper bound, we differentiate $l$ at both side of equation (\ref{e51}), then
$$
\triangle_{\omega_l}\dot{u_l}+\frac{n}{l}-\frac{1}{l}\triangle_{\omega_l}u_l\geq \frac{1}{l}(\dot{u_l}-\frac{u_l}{l}).
$$
where $\omega_l=\widetilde{\omega_l}+\sqrt{-1}\partial\overline{\partial}u_l$.

By the simple calculation, we get
$$
\triangle_{\omega_l}(\frac{u_l}{l}-n\log l)'\geq \frac{1}{l}(\frac{u_l}{l}-n\log l)'.
$$
According to the generalized maximum principle, $(\frac{u_l}{l}-n\log l)$ decrease when $l$ tends to $1$. Therefore, there exists a constant $C$ such that $u_l\leq C$. By the definition of $w_l$, we see $w_l\leq C-l\epsilon \log|s_E|^2$.
\end{proof}
\begin{lem}\label{L4.2}
There exist two constants $C$ and $a$ independent of $l$ such that $C^{-1}|s_E|^{2l\epsilon(a+1)}\widetilde{\omega_{l,E}}\leq \omega_l:=\widetilde{\omega_l}+\sqrt{-1}\partial\overline{\partial}u_l\leq C|s_E|^{-2l\epsilon(a+1)(n-1)}\widetilde{\omega_{l,E}}$.
\end{lem}
\begin{proof}
Since $\Ric(\omega_l)\geq-\frac{1}{l}\omega_l$, by Yau's Schwarz Lemma \cite{Y}, we have
$$
\triangle_{\omega_l}\log tr_{\omega_l}\widetilde{\omega_{l,E}}\geq -a\cdot tr_{\omega_l}\widetilde{\omega_{l,E}}-\frac{1}{l},
$$
where $a$ is a positive upper bound of the holomorphic bisectional curvature of $\widetilde{\omega_{l,E}}$ for $l\in [0,1]$.
Put $H=\log tr_{\omega_l}\widetilde{\omega_{l,E}}-(a+1)w_l$, then we get
$$
\triangle_{\omega_l}H\geq tr_{\omega_l}\widetilde{\omega_{l,E}}-C.
$$
By the generalized maximum principle, there exists a sequence $\{x_i\}$ such that $\lim_{i\rightarrow\infty}H(x_i)=\sup_MH$ and $\overline{\lim}_{i\rightarrow\infty}\triangle_{\omega_l}H(x_i)\leq 0$. Thus by the Lemma \ref{L4.1} we have $H\leq C$. This implies $$tr_{\omega_l}\widetilde{\omega_{l,E}}\leq \frac{C}{|s_E|^{2l\epsilon(a+1)}}.$$ Note that
$$
tr_{\widetilde{\omega_{l,E}}}\omega_l\leq \frac{1}{(n-1)!}(tr_{\omega_l}\widetilde{\omega_{l,E}})^{n-1}\cdot \frac{\omega_l^n}{\widetilde{\omega_{l,E}}^n}\leq \frac{C}{|s_E|^{2l\epsilon(a+1)(n-1)}}.
$$
Hence this Lemma is proved.
\end{proof}
According to Lemma \ref{L4.2}, we know that for any compact subset $K\subset \overline{M}\setminus (D\cup \Supp E)$, there exists a constant $C_K>0$ independent of $l$ such that $C^{-1}_K\omega_0\leq \omega_l\leq C_K\omega_0$, i.e., $|\triangle_{\omega_0}u_l|\leq C_K$. By Theorem 17.14 in \cite{GT}, we have $|u_l|_{C^{2,\lambda}}\leq C'_K$ on $K\times [\frac{1}{2},1]$. Furthermore, by the standard bootstrapping argument, for any $m>0$, $|u_l|_{C^{m,\lambda}}\leq C_{K,m}$ on $K\times [\frac{1}{2},1]$. By the standard diagonal argument and passing to a subsequence $\{l_i\}$ such that $u_{l_i}$ $C^\infty$-converge to a function on each compact $K$ when $l_i$ tends to $1$. The monotonicity of $(\frac{u_l}{l}-n\log l)$ implies that $u_l$ $C^\infty$-converge to a function on each compact $K$ when $l$ tends to $1$. Therefore, the Theorem \ref{t2} is proved.

\section{Algebraic structure of the limit space}

\subsection{Gromov-Hausdorff convergence: global convergence}
In this subsection we consider a family of manifolds $(M, \omega_l)$ on which the lower bound of  Ricci curvature can be controlled, i.e. $Ric(\omega_l)\geq -\frac{1}{l}\omega_{l}$ for
$l\in[\frac{1}{2},1)$. By Gromov precompactness theorem \cite{Cheeger}, passing to a subsequence $l_i \rightarrow 1$
and fix $x_0\in \overline{M}_{reg}\setminus D$, we may assume that
$$
  (M, \omega_{l_i}, x_0)\xrightarrow{d_{GH}} (M_1, d_1, x_1).
$$
The limit $(M_1,d_1)$ is a complete length metric space. It has a regular/singular decomposition $M_1=\mathcal{R}\cup \mathcal{S}$, a point $x\in \mathcal{R}$ iff the tangent cone at $x$ is the Euclidean space $\mathbb{R}^{2n}$. The following lemma is the same as Lemma 3.3 in \cite{NTZ}.
\begin{lem}\label{l4.4}
There is a sufficiently small constant $\delta>0$ such that for any $l\in[\frac{1}{2},1)$, if a metric ball $B_{\omega_{l}}(x,r)$ satisfies
$$ \Vol(B_{\omega_{l}}(x,r))\geq (1-\delta)\Vol(B^0_{r}),$$ where $\Vol(B^0_{r})$ is the volume of a metric ball of radius $r$ in $2n$-Euclidean space,
then $$\Ric(\omega_{l})\leq (2n-1)r^{-2}\omega_{l}, \ in \ B_{\omega_{l}}(x, \delta r).$$
\end{lem}
\begin{lem}\label{l4.5}
The regular set $\mathcal{R}$ is open in the limit space $(M_T, d_T, x_T)$.
\end{lem}
\begin{proof}
If $x\in \mathcal{R}$, then by Colding's volume convergence theorem \cite{Co} , there exists $r=r(x)>0$ such that $\mathcal{H}^{2n}(B_{d_1}(x,r))\geq (1-\frac{\delta}{2})\Vol(B^0_{r})$, where $\mathcal{H}^{2n}$ denotes the Hausdorff measure. Suppose $x_i\in M$ satisfying $x_i\xrightarrow{d_{GH}} x$, then by the volume convergence theorem again, $\Vol_{\omega_{l_i}}(B_{\omega_{l_i}}(x_i,r))\geq (1-\delta)\Vol(B^0_{r})$ for sufficiently large $i$. According to Lemma \ref{l4.4} and Anderson's harmonic radius estimate \cite{An}, there is a constant $\delta'=\delta'(\alpha)$ for any $0<\alpha<1$ such that the $C^{1,\alpha}$ harmonic radius at $x_i$ is bigger than $\delta'\delta r$. Passing to the limit, it gives a harmonic coordinate on $B_{d_1}(x,\delta'\delta r)$. This implies $B_{d_1}(x,\delta'\delta r)\subset \mathcal{R}$. So $\mathcal{R}$ is open with a $C^{1,\alpha}$ K\"{a}hler metric $\overline{\omega_1}$; moreover $\omega_{l_i}$ converges to $\overline{\omega_1}$ in $C^{1,\alpha}$ topology on $\mathcal{R}$.
\end{proof}
Since $\mathcal{R}$ is dense in $M_1$, so we have the following Lemma.
\begin{lem}
$(M_1,d_1)=\overline{(\mathcal{R}, \overline{\omega_1}})$, the metric completion of $(\mathcal{R},\overline{\omega_1})$.
\end{lem}
\begin{lem}
$\mathcal{R}$ is geodesically convex in $M_1$ in the sense that any minimal geodesic with endpoints in $\mathcal{R}$ lies in $\mathcal{R}$.
\end{lem}
\begin{proof}
It is simply a consequence of Colding-Naber's H\"{o}lder continuity of tangent cones along a geodesic in $M_1$ \cite{CN}. Actually, if $x, y\in \mathcal{R}$, then for any minimal geodesic connecting $x$ and $y$, a neighborhood of endpoints lies in $\mathcal{R}$, so the geodesic will never touch the singular set.
\end{proof}
Let $\overline{D}$ be any divisor in $\overline{M}$ such that $D \cup \mathcal{S}_{\overline{M}} \subset \overline{D}$. Define the Gromov-Hausdorff limit of $\overline{D}$
$$
\overline{ D_{1}}:=\{x\in M_1|there \ exists \ x_i\in \overline{D} such \ that \ x_i\xrightarrow{d_{GH}} x\}.
$$
\begin{prop}\label{Ana}
$(M_1,d_1)$ is isometric to $\overline{(\overline{M}\backslash \overline{D},\omega_1)}$, where $\omega_1$ is defined as Theorem \ref{t2}.
\end{prop}
\begin{proof}
First, we prove the following Claim.
\begin{claim}
$\overline{D_1}\backslash \mathcal{S}$ is a subvariety of dimension $(n-1)$ if it is not empty.
\end{claim}
\begin{proof}
Let $x\in \overline{D_1}\backslash \mathcal{S}$ and $x_i\in \overline{D}$ such that $ x_i\xrightarrow{d_{GH}} x$. By the $C^{1,\alpha}$ convergence of $\omega_{l_i}$ around $x$, there are $C, r>0$ independent of $i$ and a sequence of harmonic coordinates in $B_{\omega_{l_i}}(x_i,r)$
such that $C^{-1}\omega_E\leq \omega_{l_i}\leq C\omega_E$ where $\omega_E$ is the Euclidean metric in the coordinates. Furthermore, according to Lemma 3.11 \cite{TZ2}, any $x_i\in M$ converging to $x$ has a holomorphic coordinate $(z_i^1, z_i^2, \cdot\cdot\cdot, z_i^n)$ on $B_{\omega_{l_i}}(x_i,r)$ such that $C^{-1}\omega_E(\frac{\partial}{\partial z_i^k}, \frac{\partial}{\partial \bar{z}_i^l})\leq \omega_{l_i}(\frac{\partial}{\partial z_i^k}, \frac{\partial}{\partial \bar{z}_i^l})\leq C\omega_E(\frac{\partial}{\partial z_i^k}, \frac{\partial}{\partial \bar{z}_i^l})$. Since the total volume of
$\overline{D}$ is uniformly bounded for any $\omega_{l_i}$, the local analytic $\overline{D}\cap B_{\omega_{l_i}}(x_i,r)$ have a uniform bound of degree and so converge to an analytic set $\overline{D_1 }\cap B_{d_1}(x,r)$.
\end{proof}
From the above Claim we know that $dim_{\mathbb{R}}(\overline{D_1})=dim_{\mathbb{R}}(\mathcal{S}\cup (\overline{D_1}\backslash \mathcal{S}))\leq 2n-2$. By the argument of \cite{RZ}, $(M_1\backslash \overline{D_1}, \overline{\omega_1})$ homeomorphic and locally isometric to $(\overline{M}\backslash \overline{D},\omega_1)$. Since $M_1$ is a length space and $dim_{\mathbb{R}}(\overline{D_1})\leq 2n-2$, $(M_1\backslash \overline{D_1}, \overline{\omega_1})$ isometric to $(\overline{M}\backslash \overline{D},\omega_1)$. Furthermore, we have
$$
 (M_1,d_1)=\overline{(M_1\backslash \overline{D_1},\overline{\omega_1})}=\overline{(\overline{M}\backslash \overline{D},\omega_1)}.
$$
\end{proof}
A direct corollary is
\begin{cor}
$(M,\omega_l, x_0)$ converges globally to $(M_1,d_1, x_1)$ in the Gromov-Hausdorff topology as $l\rightarrow l$.
\end{cor}
\begin{cor}
Let $M_{reg}$=$\overline{M}_{reg}\setminus D$, then $\omega_1$ is smooth on $M_{reg}$. $(M_1,d_1)$ is isometric to $\overline{(M_{reg} ,\omega_1)}$.
\end{cor}
\begin{proof}
Note that $M_{reg}\backslash  (\overline{M}\backslash \overline{D})=M_{reg}\cap \overline{D}$ has real codimension larger than $2$ in
$(M_{reg},\omega_1)$. So $\overline{M}\backslash \overline{D}$ is dense in $M_{reg}$. We conclude
$$
 (M_1,d_1)=\overline{(\overline{M}\backslash \overline{D},\omega_1)}=\overline{(M_{reg},\omega_1)}.
$$
\end{proof}
\begin{prop}
$M_{reg}=\mathcal{R}$, the regular set of $M_1$.
\end{prop}
\begin{proof}
Since $M_{reg}$ has smooth structure in $M_1$, we have $M_{reg}\subset \mathcal{R}$. Next we show the converse. We argue by contradiction. Suppose $p\in \mathcal{R}\setminus M_{reg}$,
then there exists a family of points $p_l\in M_{sing}$ such that $p_l\xrightarrow{d_{GH}} p$, where $M_{sing}=(\Phi^{-1}(\Phi(\overline{M})_{sing}))\backslash D$. By $C^{1,\alpha}$ convergence on $\mathcal{R}$, there exist
$C,r>0$ independent of $l$ and a sequence of harmonic coordinates on $B_{\omega_{l}}(p_l,r)$ such that $C^{-1}\omega_E\leq \omega_{l}\leq C\omega_E$ where $\omega_E$ is the Euclidean metric in this coordinate. Furthermore, the sequence of harmonic coordinate can be perturbed to a holomorphic coordinate on $B_{\omega_{l}}(p_l,r)$ \cite{TZ2}. Denote $m=dim_{\mathbb{C}}(M_{sing})$. Then
\begin{align*}
 & \Vol_{\omega_{l}}(M_{sing}\cap B_{\omega_{l}}(p_l,r))=\int_{M_{sing}\cap B_{\omega_{l}}(p_l,r)}\omega_{l}^{m} \\
 & \qquad \geq \int_{M_{sing}\cap B_{\omega_E}(C^{-\frac{1}{2}}r)}(C^{-1}\omega_E)^m
\end{align*}
which has a uniform lower bound. However, this contradicts with the
degeneration of the limit metric $\eta_1$ along $M_{sing}$:
\begin{align*}
& \Vol_{\omega_{l}}(M_{sing}\cap B_{\omega_{l}}(p_l,r))\leq \Vol_{\omega_{l}}(M_{sing}) \\
& \qquad = \int_{M_{sing}}\omega_{l}^{m}=\int_{M_{sing}}((1-l)\omega_0+l\eta_1)^m
\end{align*}
which tends to $0$ as $l\rightarrow 1$, where the last equality bases on a Lemma (\cite{Ko84} P410). So we have $M_{reg}\supset \mathcal{R}$.
\end{proof}

\subsection{$L^{\infty}$ estimate and gradient estimate to holomorphic sections}
In this subsection we obtain the $L^{\infty}$ estimate and gradient estimate to holomorphic section $s\in H^0(\mathcal{R},k(K_{\overline{M}}+D))$. $h=\omega_1^{-nk}$ is chosen as the Hermitian metric of line bundle $k(K_{\overline{M}}+D)$, where $k\in \mathbb{Z}$. The curvature form $\Theta_h$ of Hermitian metric $h=\omega_1^{-nk}$ is $k\omega_1$. By Lemma \ref{Bo}, we have the following formulas.
\begin{lem}\label{Bo1}
For $s\in H^0(\mathcal{R},k(K_{\overline{M}}+D))$, there exists a constant $C$ such that
$$
\triangle_{\omega_1}|s|^2=|\nabla s|^2-kn|s|^2
$$
and
$$
\triangle_{\omega_1}|\nabla s|^2\geq |\overline{\nabla}\nabla s|^2+|\nabla\nabla s|^2-Ck|\nabla s|^2-k\nabla_j(\omega_1)_{i\bar{j}}\langle s,\nabla_{\bar{i}}\bar{s}\rangle.
$$
\end{lem}
\begin{proof}
Since on $\mathcal{R}$, $\Ric(\omega_1)=-\omega_1$. So these formulas are directly derived from Lemma \ref{Bo}.
\end{proof}
In order to applying Moser iteration, the Sobolev inequality on $\mathcal{R}$ is needed. The following two Lemmas are due to Song (Lemma 3.7 and 4.6 \cite{So}).
\begin{lem}\label{cf}
There is a family of cut-off functions $\rho_\epsilon\in C_0^\infty(\mathcal{R})$ with $0<\rho_\epsilon<1$ such that $\rho^{-1}_\epsilon(1)$ forms an exhaustion of $\mathcal{R}$ and
$$
\int_{\mathcal{R}}|\nabla \rho_\epsilon|^2\omega_1^n\rightarrow 0.
$$
\end{lem}
\begin{lem}\label{So}
Fix any $0<r<R$, the Sobolev constant on $B_{\omega_l}(x,r)$ is uniformly bounded by a constant $C_S$ depending on upper bound of $R$, $R^{-1}$ and $(R-r)^{-1}$. More precisely, for any $l\in [\frac{1}{2},1)$ and $f\in C_0^1(B_{\omega_l}(x,r))$,
$$
C_S\bigg(\int_{B_{\omega_l}(x,r)}|f|^{\frac{2n}{n-1}}\omega_l^n\bigg)^{\frac{n-1}{n}}\leq \int_{B_{\omega_l}(x,r)}(|f|^2+|\nabla f|^2)\omega^n_l.
$$
\end{lem}
Fix $0<r<R$ such that $B_{\omega_1}(x,r)\subset B_{\omega_1}(x,2r)\subset B_{\omega_1}(x,R)$.
\begin{lem}\label{So2}
If $f\in C_0^1(B_{\omega_1}(x,r)\cap \mathcal{R})$, then there exists a constant $C$ depending on $R$, $R^{-1}$ and $(R-r)^{-1}$ such that
$$
C\bigg(\int_{B_{\omega_1}(x,r)\cap \mathcal{R}}|f|^{\frac{2n}{n-1}}\omega_1^n\bigg)^{\frac{n-1}{n}}\leq \int_{B_{\omega_1}(x,r)\cap \mathcal{R}}(|f|^2+|\nabla f|^2)\omega^n_1.
$$
\end{lem}
\begin{proof}
Let $f_\epsilon=\rho_\epsilon f$, where $\rho_\epsilon$ is constructed as Lemma \ref{cf} and $\Omega_\epsilon=\Supp f_\epsilon$. Then $\omega_l$ uniformly converge to $\omega_1$ on $\Omega_\epsilon$ as $l$ tends to $1$ for a fixed $\epsilon$. Therefore $\Omega_\epsilon\subset B_{\omega_l}(x,r)$ for $l$ sufficiently close to $1$. By Lemma \ref{So}, we have
$$
C_S\bigg(\int_{B_{\omega_l}(x,r)}|f_\epsilon|^{\frac{2n}{n-1}}\omega_l^n\bigg)^{\frac{n-1}{n}}\leq \int_{B_{\omega_l}(x,r)}(|f_\epsilon|^2+|\nabla f_\epsilon|^2)\omega^n_l.
$$
Let $l\rightarrow 1$, the above inequality gives
$$
C_S\bigg(\int_{B_{\omega_1}(x,r)}|f_\epsilon|^{\frac{2n}{n-1}}\omega_1^n\bigg)^{\frac{n-1}{n}}\leq \int_{B_{\omega_1}(x,r)}(|f_\epsilon|^2+|\nabla f_\epsilon|^2)\omega^n_1.
$$
Note that by letting $\epsilon\rightarrow 0$, we get
$$
\int_{B_{\omega_1}(x,r)}|f_\epsilon|^{\frac{2n}{n-1}}\omega_1^n\rightarrow \int_{B_{\omega_1}(x,r)}|f|^{\frac{2n}{n-1}}\omega_1^n
$$
and
$$
\int_{B_{\omega_1}(x,r)}|f_\epsilon|^{2}\omega_1^n\rightarrow \int_{B_{\omega_1}(x,r)}|f|^{2}\omega_1^n.
$$
By some calculations we have
$$
\bigg|\int_{B_{\omega_1}(x,r)}|\nabla f_\epsilon|^2\omega^n_1-\int_{B_{\omega_1}(x,r)}|\nabla f|^2\omega^n_1\bigg|=\bigg|\int_{B_{\omega_1}(x,r)}\big(|\nabla\rho_\epsilon|^2|f|^2+(|\rho_\epsilon|^2|\nabla f|^2-|\nabla f|^2)\big)\omega_1^n\bigg|
$$
which tends to $0$. So this Lemma is proved.
\end{proof}
\begin{lem}\label{Com}
There exists a constant $C$ independent of $k$ such that if $s\in H^0(\mathcal{R},k(K_{\overline{M}}+D))$, then
$$
\int_{B_{\omega_1}(x,\frac{7}{4}r)\cap \mathcal{R}}|\nabla s|^2\omega_1^n\leq Ckr^{-2}\int_{B_{\omega_1}(x,2r)\cap \mathcal{R}}|s|^2\omega_1^n
$$
and
$$
\int_{B_{\omega_1}(x,\frac{7}{4}r)\cap \mathcal{R}}(|\bar{\nabla}\nabla s|^2+|\nabla\nabla s|^2)\omega_1^n\leq Ck^2r^{-4}\int_{B_{\omega_1}(x,2r)\cap \mathcal{R}}|s|^2\omega_1^n.
$$
\end{lem}
\begin{proof}
Let $\vartheta\in C_0^\infty(B_{\omega_1}(x,\frac{15}{8}r)\cap \mathcal{R})$ be any cut-off function such that $0\leq \vartheta\leq 1$, $|\nabla\vartheta|\leq 10r^{-2}$ and $\vartheta=1$ on $B_{\omega_1}(x,\frac{7}{4}r)\cap \mathcal{R}$, then by Bochner formula we have
$$
\int_{\mathcal{R}}\vartheta^2|\nabla s|^2\omega_1^n=nk\int_{\mathcal{R}}\vartheta^2|s|^2\omega_1^n+\int_{\mathcal{R}}\vartheta^2\triangle|s|^2\omega_1^n.
$$
Note that
$$
\int_{\mathcal{R}}\vartheta^2\triangle|s|^2\omega_1^n=-2\int_{\mathcal{R}}\vartheta\nabla_{\bar{i}}\vartheta\langle\nabla_is, \bar{s}\rangle\omega_1^n\leq \frac{1}{2}\int_{\mathcal{R}}\vartheta^2|\nabla s|^2\omega_1^n+2\int_{\mathcal{R}}|\nabla \vartheta|^2|s|^2\omega_1^n.
$$
Therefore,
$$
\int_{B_{\omega_1}(x,\frac{7}{4}r)\cap \mathcal{R}}|\nabla s|^2\omega_1^n\leq Ckr^{-2}\int_{B_{\omega_1}(x,2r)\cap \mathcal{R}}|s|^2\omega_1^n.
$$
For the second inequality, also by the Bochner formula
$$
\int_{\mathcal{R}}\vartheta^2(|\bar{\nabla}\nabla s|^2+|\nabla\nabla s|^2)\omega_1^n\leq \int_{\mathcal{R}}\vartheta^2(\triangle|\nabla s|^2+Ck|\nabla s|^2+k\nabla_j(\omega_1)_{i\bar{j}}\langle s, \nabla_{\bar{i}}\bar{s}\rangle)\omega_1^n.
$$
Note that
$$
\int_{\mathcal{R}}\vartheta^2\triangle|\nabla s|^2\omega_1^n=-2\int_{\mathcal{R}}\vartheta\nabla_i\vartheta\nabla_{\bar{i}}|\nabla s|^2\omega_1^n\leq\frac{1}{4}\int_{\mathcal{R}}\vartheta^2(|\bar{\nabla}\nabla s|^2+|\nabla\nabla s|^2)\omega_1^n+C\int_{\mathcal{R}}|\nabla\vartheta|^2+|\nabla s|^2\omega_1^n
$$
and
\begin{align*}
\int_{\mathcal{R}}k\vartheta^2\nabla_j(\omega_1)_{i\bar{j}}\langle s, \nabla_{\bar{i}}\bar{s}\rangle\omega_1^n & =-k\int_{\mathcal{R}}\vartheta^2(\omega_1)_{i\bar{j}}(\langle\nabla_js,\nabla_{\bar{i}}\bar{s}\rangle+\langle s,\nabla_{\bar{j}}\nabla_{\bar{i}}\bar{s}\rangle)\omega_1^n-2k\int_{\mathcal{R}}\vartheta\nabla_j\vartheta(\omega_1)_{i\bar{j}}\langle s, \nabla_{\bar{i}}\bar{s}\rangle\omega_1^n\\
& \leq \frac{1}{4}\int_{\mathcal{R}}\vartheta^2(|\bar{\nabla}\nabla s|^2+|\nabla\nabla s|^2)\omega_1^n+Ck\int_{\mathcal{R}}\vartheta^2|s|^2\omega^n_1+Ck\int_{\mathcal{R}}\vartheta^2|\nabla s|^2\omega^n_1\\
&\qquad +Ck\int_{\mathcal{R}}|\nabla\vartheta|^2|\nabla s|^2\omega^n_1
\end{align*}
Summing up these estimates we have
$$
\int_{B_{\omega_1}(x,\frac{7}{4}r)\cap \mathcal{R}}(|\bar{\nabla}\nabla s|^2+|\nabla\nabla s|^2)\omega_1^n\leq Ckr^{-2}\int_{B_{\omega_1}(x,2r)\cap \mathcal{R}}|\nabla s|^2\omega_1^n+Ck\int_{B_{\omega_1}(x,2r)\cap \mathcal{R}}|s|^2\omega_1^n.
$$
Applying the first inequality we obtain the second estimate.
\end{proof}
\begin{prop}\label{L0a}
There exists a constant $C(R,r)$ independent of $k$ such that if $s\in H^0(\mathcal{R},k(K_{\overline{M}}+D))$, then
$$
\sup_{B_{\omega_1}(x,r)\cap \mathcal{R}}|s|^2\leq C(R,r)k^nr^{-2n}\int_{B_{\omega_1}(x,2r)\cap \mathcal{R}}|s|^2\omega_1^n
$$
and
$$
\sup_{B_{\omega_1}(x,r)\cap \mathcal{R}}|\nabla s|^2\leq  C(R,r)k^{n+1}r^{-2n-2}\int_{B_{\omega_1}(x,2r)\cap \mathcal{R}}|s|^2\omega_1^n.
$$
\end{prop}
\begin{proof}
Choose a cut-off function $\vartheta\in C_0^\infty(B_{\omega_1}(x,2r)\cap \mathcal{R})$. Then for any $p\geq\frac{n}{n-1}$, by Lemma \ref{Bo1}, we have
\begin{align*}
\int_{\mathcal{R}}\vartheta^2|\nabla|s|^p|^2\omega_{1}^{n}&=\frac{p^2}{4(p-1)}\int_{\mathcal{R}}\vartheta^2\nabla_i|s|^{2(p-1)}\nabla_{\bar{i}}|s|^2\omega_{1}^{n}\\
&= \frac{p^2}{4(p-1)}\int_{\mathcal{R}}(-\vartheta^2|s|^{2(p-1)}\triangle_{\omega_1}|s|^2-2\vartheta \cdot \nabla_i\vartheta \cdot|s|^{2(p-1)} \cdot\nabla_{\bar{i}}|s|^2)\omega_{1}^{n}\\
&\leq \frac{p^2}{4(p-1)}\int_{\mathcal{R}} -\vartheta^2|s|^{2(p-1)}|\nabla s|^2\omega_{1}^{n}+nk\frac{p^2}{4(p-1)}\int_{\mathcal{R}} \vartheta^2|s|^{2p}\omega_{1}^{n} \\
& \qquad +\frac{p^2}{4(p-1)}\int_{\mathcal{R}} \vartheta\cdot|\nabla\vartheta|\cdot|s|^{2(p-1)}\cdot|\nabla s|\cdot|s|\omega_{1}^{n}.
\end{align*}
By Cauchy-Schwarz inequality,
$$
\int_{\mathcal{R}} \vartheta\cdot|\nabla\vartheta|\cdot|s|^{2(p-1)}\cdot|\nabla s|\cdot|s|\omega_{1}^{n}\leq \int_{\mathcal{R}}\vartheta^2|s|^{2(p-1)}|\nabla s|^2\omega_1^n+\frac{1}{4}\int_{\mathcal{R}}|\nabla\vartheta|^2|s|^{2p}\omega_1^n.
$$
Therefore
$$
\int_{\mathcal{R}}\vartheta^2|\nabla|s|^p|^2\omega_{1}^{n}\leq Cpk\int_{\mathcal{R}}(\vartheta^2+|\nabla \vartheta|^2)|s|^{2p}\omega_{1}^{n}.
$$
By Lemma \ref{So2}
\begin{equation}\label{e4.18}
\bigg(\int_{\mathcal{R}} (\vartheta|s|^p)^{\frac{2n}{n-1}}\omega_{1}^{n}\bigg)^{\frac{n-1}{n}}\leq Cpk\int_{\mathcal{R}}(\vartheta^2+|\nabla\vartheta|^2)|s|^{2p}\omega_{1}^{n}.
\end{equation}\label{e5.1}
Put $p_j=\nu^{j+1}$ for $j\geq 0$, where $\nu=\frac{n}{n-1}$. Define a family of radius inductively by $r_0=\frac{3}{2}r$ and $r_j=r_{j-1}-2^{-j-1}r$. $B_j$ is denoted by $B_{\omega_1}(x,r_j)\cap \mathcal{R}$. We choose a family of cut-off functions $\vartheta_j\in C^\infty_0(B_j)$ such that
$$
0\leq \vartheta_j\leq 1,\ |\nabla\vartheta_j|\leq 2^{j+2}r^{-1} \ and \ \vartheta_j=1 \ on \ B_{j+1}.
$$
Thus (\ref{e5.1}) gives, by setting $\vartheta=\vartheta_j$
$$
\bigg(\int_{B_{j+1}}|s|^{2p_{j+1}}\omega_{1}^n\bigg)^{\frac{1}{p_{j+1}}}\leq (Cp_jk)^{\frac{1}{p_j}}4^{\frac{j}{p_j}}r^{-\frac{2}{p_j}}\bigg(\int_{B_j}|s|^{2p_j}\omega_{1}^n\bigg)^{\frac{1}{p_j}}.
$$
By the iteration argument, we see
$$
\sup_{B_{\omega_{1}}(x,r)\cap \mathcal{R}}|s|^2\leq Ck^{n-1}r^{-2(n-1)}\bigg(\int_{B_0}|s|^{\frac{2n}{n-1}}\omega_1^n\bigg)^{\frac{n}{n-1}}\leq Ck^{n-1}r^{-2(n-1)}\int_{B_0}\Big(|s|^2+|\nabla s|^2\Big)\omega_1^n.
$$
According to Lemma \ref{Com}, we get the first estimate.

Next we prove the second inequality. Let $\vartheta$ and $p$ as above. By Lemma \ref{Bo1}, we have
\begin{align*}
\int_{\mathcal{R}}\vartheta^2|\nabla|\nabla S|^p|^2\omega_1^n &=\frac{p^2}{4(p-1)}\int_{\mathcal{R}}\vartheta^2\cdot \nabla_i|\nabla s|^{2(p-1)}\cdot \nabla_{\bar{i}}|\nabla s|^2\omega_1^n\\
&= \frac{p^2}{4(p-1)}\int_{\mathcal{R}}(-\vartheta^2|\nabla s|^{2(p-1)}\triangle|\nabla s|^2-2\vartheta\cdot \nabla_i\vartheta\cdot|\nabla s|^{2(p-1)}\nabla_{\bar{i}}|\nabla s|^2) \omega_1^n \\
&\leq \frac{p^2}{4(p-1)}\int_{\mathcal{R}}\bigg(-\vartheta^2|\nabla s|^{2(p-1)}(|\bar{\nabla}\nabla s|^2+|\nabla\nabla s|^2)+k\nabla_j(\omega_1)_{i\bar{j}}\langle s,\nabla_{\bar{i}}\bar{s}\rangle\cdot\vartheta^2|\nabla s|^{2(p-1)}\\
& \qquad +Ck\vartheta^2|\nabla s|^{2p}-2\vartheta\cdot\nabla_i\vartheta\cdot|\nabla s|^{2(p-1)}\cdot\nabla_{\bar{i}}|\nabla s|^2\bigg)\omega_1^n
\end{align*}
The term $\int_{\mathcal{R}}k\nabla_j(\omega_1)_{i\bar{j}}\langle s,\nabla_{\bar{i}}\bar{s}\rangle\cdot\vartheta^2|\nabla s|^{2(p-1)}\omega_1^n$ can be estimate by integration by parts as follows
\begin{align*}
&\int_{\mathcal{R}}k\nabla_j(\omega_1)_{i\bar{j}}\langle s,\nabla_{\bar{i}}\bar{s}\rangle\cdot\vartheta^2|\nabla s|^{2(p-1)}\omega_1^n\\
=& -k\int_{\mathcal{R}}(\omega_1)_{i\bar{j}}\Big(\vartheta^2|\nabla s|^{2(p-1)}\big(\langle\nabla_j s,\nabla_{\bar{i}}\bar{s}\rangle+\langle s,\nabla_j\nabla_{\bar{i}}\bar{s}\rangle\big)+(p-1)\vartheta^2|\nabla s|^{2(p-2)}\nabla_{\bar{j}}|\nabla s|^2\langle s, \nabla_{\bar{i}}\bar{s}\rangle\\
& + 2\vartheta\nabla_j\vartheta|\nabla s|^{2(p-1)}\langle s, \nabla_{\bar{i}}\bar{s}\rangle\Big)\\
\leq &  \frac{1}{2}\int_{\mathcal{R}}\vartheta^2|\nabla s|^{2(p-1)}(|\bar{\nabla}\nabla s|^2+|\nabla\nabla s|^2)\omega_1^n+C(p-1)^2k^2\int_{\mathcal{R}}\vartheta^2|s|^2|\nabla s|^{2(p-1)}\omega_1^n \\
& + Ck\int_{\mathcal{R}}\Big(|\nabla\vartheta|^2|s|^2|\nabla s|^{2(p-1)}+\vartheta^2|\nabla s|^{2p}\Big)\omega_1^n.
\end{align*}
Note that
$$
-2\int_{\mathcal{R}}\vartheta\cdot\nabla_i\vartheta\cdot|\nabla s|^{2(p-1)}\cdot\nabla_{\bar{i}}|\nabla s|^2\omega_1^n\leq \frac{1}{2}\int_{\mathcal{R}}\vartheta^2|\nabla s|^{2(p-1)}(|\bar{\nabla}\nabla s|^2+|\nabla\nabla s|^2)\omega_1^n+C\int_{\mathcal{R}}|\nabla \vartheta|^2|\nabla s|^{2p}\omega_1^n.
$$
Summing up these estimates we conclude
$$
\int_{\mathcal{R}}\vartheta^2|\nabla|\nabla S|^p|^2\omega_1^n\leq Cp^3k\int_{\mathcal{R}}\Big(k\vartheta^2|\nabla s|^{2(p-1)}|s|^2+|\nabla\vartheta|^2|s|^2|\nabla s|^{2(p-1)}+\vartheta^2|\nabla s|^{2p}+|\nabla s|^{2p}|\nabla\vartheta|^2\Big)\omega_1^n.
$$
Applying the Lemma \ref{So2}

\begin{align*}
\Big(\int_{\mathcal{R}}\Big(\vartheta|\nabla s|^p\Big)^{\frac{2n}{n-1}}\omega_1^n\Big)^{\frac{n-1}{n}} & \leq Cp^3k\int_{\mathcal{R}}\Big(k\vartheta^2|\nabla s|^{2(p-1)}|s|^2+|\nabla\vartheta|^2|s|^2|\nabla s|^{2(p-1)}\\
& +\vartheta^2|\nabla s|^{2p}+|\nabla s|^{2p}|\nabla\vartheta|^2\Big)\omega_1^n.
\end{align*}

Put $p_j=\nu^{j+1}$ for $j\geq 0$, where $\nu=\frac{n}{n-1}$. Define a family of radius inductively by $r_0=\frac{3}{2}r$ and $r_j=r_{j-1}-2^{-j-1}r$. $B_j$ is denoted by $B_{\omega_1}(x,r_j)\cap \mathcal{R}$. We choose a family of cut-off functions $\vartheta_j\in C^\infty_0(B_j)$ such that
$$
0\leq \vartheta_j\leq 1,\ |\nabla\vartheta_j|\leq 2^{j+2}r^{-1} \ and \ \vartheta_j=1 \ on \ B_{j+1}.
$$
By setting $\vartheta=\vartheta_j$, the above inequality gives
\begin{equation}\label{So3}
\Big(\int_{B_{j+1}}|\nabla s|^{2p_{j+1}}\omega_1^n\Big)^{\frac{n}{n-1}}\leq Cp_j^3k4^jr^{-2}\int_{B_j}\Big(|\nabla s|^{2p_j}+k|\nabla s|^{2(p_j-1)}|s|^2\Big)\omega_1^n.
\end{equation}
Case 1: If $\Big(\int_{B_{j}}|\nabla s|^{2p_{j}}\omega_1^n\Big)^{\frac{1}{p_j}}\geq k\Big(\int_{B_{j}}|s|^{2p_{j}}\omega_1^n\Big)^{\frac{1}{p_j}}$ for all $j\geq 0$. Then
$$
\int_{B_j}k|\nabla s|^{2(p_j-1)}|s|^2\omega_1^n\leq k\Big(\int_{B_{j}}|\nabla s|^{2p_{j}}\omega_1^n\Big)^{\frac{p_j-1}{p_j}}\Big(\int_{B_{j}}|s|^{2p_{j}}\omega_1^n\Big)^{\frac{1}{p_j}}\leq \int_{B_{j}}|\nabla s|^{2p_{j}}\omega_1^n.
$$
Then (\ref{So3}) gives
$$
\Big(\int_{B_{j+1}}|\nabla s|^{2p_{j+1}}\omega_1^n\Big)^{\frac{1}{p_{j+1}}}\leq C(kr^{-2})^{\frac{1}{p_j}}4^{\frac{j}{p_j}}p_j^{\frac{3}{p_j}}\Big(\int_{B_{j}}|\nabla s|^{2p_{j}}\omega_1^n\Big)^{\frac{1}{p_j}}
$$
By iteration argument we get
$$
\sup_{B_{\omega_1}(x,r)\cap \mathcal{R}}|\nabla s|^2\leq C(kr^{-2})^{n-1} \Big(\int_{B_0}|\nabla s|^{\frac{2n}{n-1}}\omega_1^n\Big)^{\frac{n-1}{n}}.
$$
By Lemma \ref{So2} and a cut-off argument, we have
$$
\Big(\int_{B_0}|\nabla s|^{\frac{2n}{n-1}}\omega_1^n\Big)^{\frac{n-1}{n}}\leq C\int_{B_{\omega_1}(x,\frac{7}{4}r)\cap \mathcal{R}}\big(|\bar{\nabla}\nabla s|^2+|\nabla\nabla s|^2+r^{-2}|\nabla s|^2\big)\omega_1^n
$$
According to Lemma \ref{Com} we get
$$
\sup_{B_{\omega_1}(x,r)\cap \mathcal{R}}|\nabla s|^2\leq C(kr^{-2})^{n+1}\int_{B_{\omega_1}(x,2r)\cap \mathcal{R}}|s|^2\omega_1^n.
$$
Case2: There exists $j_0$ such that $\Big(\int_{B_{j}}|\nabla s|^{2p_{j}}\omega_1^n\Big)^{\frac{1}{p_j}}\geq k\Big(\int_{B_{j}}|s|^{2p_{j}}\omega_1^n\Big)^{\frac{1}{p_j}}$ for all $j>j_0$, but
$$
\Big(\int_{B_{j_0}}|\nabla s|^{2p_{j_0}}\omega_1^n\Big)^{\frac{1}{p_{j_0}}}< k\Big(\int_{B_{j_0}}|s|^{2p_{j_0}}\omega_1^n\Big)^{\frac{1}{p_{j_0}}}.
$$
Then
$$
\int_{B_{j_0}}k|\nabla s|^{2(p_{j_0}-1)}|s|^2\omega_1^n\leq k\Big(\int_{B_{j_0}}|\nabla s|^{2p_{j_0}}\omega_1^n\Big)^{\frac{p_{j_0}-1}{p_{j_0}}}\Big(\int_{B_{j_0}}|s|^{2p_{j_0}}\omega_1^n\Big)^{\frac{1}{p_{j_0}}}\leq k^{p_{j_0}}\int_{B_{j_0}}|s|^{2p_{j_0}}\omega_1^n.
$$
By the iteration argument and (\ref{So3}), we have
$$
\sup_{B_{\omega_1}(x,r)\cap \mathcal{R}}|\nabla s|^2\leq Ck(kr^{-2})^{\frac{n}{p_{j_0}}}\Big(\int_{B_{j_0}}|s|^{2p_{j_0}}\omega_1^n\Big)^{\frac{1}{p_{j_0}}}.
$$
The supermum of $|\nabla s|$ follows from
$$
\Big(\int_{B_{j_0}}|s|^{2p_{j_0}}\omega_1^n\Big)^{\frac{1}{p_{j_0}}}\leq \Big(\sup_{B_{j_0}}|s|\Big)^{\frac{2p_{j_0}-2}{p_{j_0}}}\Big(\int_{B_{j_0}}|s|^{2}\omega_1^n\Big)^{\frac{1}{p_{j_0}}}\leq C(kr^{-2})^{n-\frac{n}{p_{j_0}}}\int_{B_{j_0}}|s|^{2}\omega_1^n.
$$
Case 3: If $\Big(\int_{B_{j}}|\nabla s|^{2p_{j}}\omega_1^n\Big)^{\frac{1}{p_j}}\leq k\Big(\int_{B_{j}}|s|^{2p_{j}}\omega_1^n\Big)^{\frac{1}{p_j}}$ for infinite $i$, then
$$
\sup_{B_{\omega_1}(x,r)\cap \mathcal{R}}|\nabla s|^2\leq k\sup_{B_{\omega_1}(x,r)\cap \mathcal{R}}|s|\leq Ck^{n+1}r^{-2n}\int_{B_{\omega_1}(x,2r)\cap \mathcal{R}}|s|^2\omega_1^n.
$$
\end{proof}
\subsection{$L^2$ estimate}
In order to construct global holomorphic section on line bundle $k(K_{\overline{M}}+D)$, we need the following version of $L^2$-estimate due to Demailly (Theorem 5.1 \cite{De}).
\begin{thm}
Let $(M,\omega)$ be a $n$-dimensional complete K\"{a}hler manifold and $L$ be a holomorphic line bundle over $M$ equipped with a smooth hermitian metric such that $\Theta_h\geq \delta\omega$. Then for every $L$-value $(n,1)$ form $\tau$ satisfying
$$
\overline{\partial}\tau=0,\ \ \ \int_M|\tau|^2_{h,\omega}\omega^n<\infty,
$$
there exists a $L$-valued $(n,0)$ form $u$ such that $\bar{\partial}u=\tau$ and
$$
\int_M|u|^2_{h,\omega}\omega^n\leq \frac{1}{\delta}\int_M|\tau|^2_{h,\omega}\omega^n.
$$
\end{thm}
For the singular hermitian metric $h$ on $L$, by the approximation argument, we have
\begin{cor}
Let $(M,\omega)$ be a $n$-dimensional complete K\"{a}hler manifold and $L$ be a holomorphic line bundle over $M$ equipped with a singular hermitian metric such that $\Theta_h\geq \delta\omega$ in the current sense. Then for every $L$-value $(n,1)$ form $\tau$ satisfying
$$
\overline{\partial}\tau=0,\ \ \ \int_M|\tau|^2_{h,\omega}\omega^n<\infty,
$$
there exists a $L$-valued $(n,0)$ form $u$ such that $\bar{\partial}u=\tau$ and
$$
\int_M|u|^2_{h,\omega}\omega^n\leq \frac{1}{\delta}\int_M|\tau|^2_{h,\omega}\omega^n.
$$
\end{cor}
\begin{prop}\label{L2}
$(\mathcal{R}=M_{reg},k\omega_1)$ is a K\"{a}hler manifold (not complete). $k(K_{\overline{M}}+D)$ is a holomorphic line bundle over $\mathcal{R}$. Choosing a hermitian metric $h=\omega_1^{-nk}$, then the curvature form $\Theta_h=k\omega_1$. For any smooth $k(K_{\overline{M}}+D)$-valued $(0,1)$ form $\tau$ satisfying
$$
\bar{\partial}\tau=0,\ \ \ \ \Supp\tau\subset \mathcal{R}
$$
there exists a $k(K_{\overline{M}}+D)$-valued section $\varsigma$ such that $\bar{\partial}\varsigma=\tau$ and
$$
\int_{\mathcal{R}}|\varsigma|^2_{h}(k\omega_1)^n\leq \int_{\mathcal{R}}|\tau|^2_{h,k\omega_1}(k\omega_1)^n.
$$
\end{prop}
\begin{proof}
Since $K_{\overline{M}}+D$ is big and semi-ample over $\overline{M}$, by Lemma \ref{L1}, there exists an effective divisor $E$ on $\overline{M}$ such that $K_{\overline{M}}+D-\epsilon E$ is ample for all sufficiently small $\epsilon>0$. Let $s_E$ be the defining section of $E$ and $h_E$ be a smooth hermitian metric satisfying $\eta_1-\epsilon\Theta_E>0$, where $\eta_1$ is constructed as section 3 and $\Theta_E$ is the curvature form. We consider the following Monge-Amp\`{e}re equation
$$
(\eta_1-\epsilon\Theta_E-\sqrt{-1}\partial\overline{\partial}\log\log^2|s_D|^2+\sqrt{-1}\partial\overline{\partial}u_{1,\epsilon})^n=e^{u_{1,\epsilon}}\cdot\frac{\Omega}{|s_D|^2\log^2|s_D|^2}.
$$
Fixed a small $\alpha>0$, this equation is rewritten as
$$
((1-\alpha)\eta_1+\alpha(\eta_1-\frac{\epsilon}{\alpha}\Theta_E)-\sqrt{-1}\partial\overline{\partial}\log\log^2|s_D|^2+\sqrt{-1}\partial\overline{\partial}u_{1,\epsilon})^n=e^{u_{1,\epsilon}}\cdot\frac{\Omega}{|s_D|^2\log^2|s_D|^2}.
$$
By the same argument of subsection 5.5, we know that $\omega_{1,\epsilon}£º=\eta_1-\epsilon\Theta_E-\sqrt{-1}\partial\overline{\partial}\log\log^2|s_D|^2+\sqrt{-1}\partial\overline{\partial}u_{1,\epsilon}$ $C_{loc}^\infty(M_{reg})$-converge to $\omega_1$ as $\epsilon$ tends to $0$.  Now we define a family of hermitian metric $$h_\epsilon=e^{-ku_{1,\epsilon}}\Big(\frac{\Omega}{|s_D|^2\log^2|s_D|^2}\Big)^{-k}e^{-\epsilon k\log|s_E|^2}.$$ By a direct calculation, $\Ric(h_\epsilon)\geq k\omega_{1,\epsilon}$ in the current sense. $\tau$ has compact support and
$$
\lim_{\epsilon\rightarrow 0}\int_{M=\overline{M}\backslash D}|\tau|^2_{h_\epsilon,k\omega_{1,\epsilon}}(k\omega_{1,\epsilon})^n=\int_{M}|\tau|^2_{h,k\omega}(k\omega_1)^n<\infty.
$$
So by the above corollary, there exists $\varsigma_\epsilon$ on $M$ such that
$$
\bar{\partial}\varsigma_\epsilon=\tau,\ \ \ \ \ \int_{M}|\varsigma_\epsilon|^2_{h_\epsilon}(k\omega_{1,\epsilon})^n\leq \int_{M}|\tau|^2_{h_\epsilon,k\omega_{1,\epsilon}}(k\omega_{1,\epsilon})^n
$$
for each $\epsilon$. This also implies
$$
\int_{M}|\varsigma_\epsilon|^2_{h}(k\omega_{1})^n<\infty.
$$
Hence we can take a subsequence of $\varsigma_\epsilon$ converging weakly in $L^2(M,(k\omega_1)^n)$ to $\varsigma$ and
$$
\bar{\partial}\varsigma=\tau, \ \ \ \ \ \ \int_{M}|\varsigma_\epsilon|^2_{h}(k\omega_{1})^n\leq \int_{M}|\tau|^2_{h,k\omega_{1}}(k\omega_{1})^n
$$
on $M$. The proof is complete after pushing $\varsigma$ to $M_{reg}$.
\end{proof}
\subsection{local separation of points}
Recall that $\Phi:\overline{M}\rightarrow \Phi(\overline{M})$ is defined as in section 3. Naturally, $\Phi$ induce a map $\Phi:\mathcal{R}\rightarrow \Phi(\overline{M}\backslash D)$. If $s\in H^0(\mathcal{R},k(K_{\overline{M}}+D))$, then by Proposition \ref{L0a}, we know that $s$ is local bounded and local Lipschitz. So $s$ can be continuous extended to the limit space $M_1$. Furthermore, the map $\Phi_1:(\mathcal{R},\omega_1)\rightarrow (\Phi(\overline{M}\backslash D),\omega_{FS})$ defined by $\Phi$ can be continuously extend to $\Phi_1:(M_1,d_1)\rightarrow (\Phi(\overline{M}\backslash D),\omega_{FS})$. This subsection is devoted to demonstrate that this map is injective. First we recall some notations and results which originate from \cite{DS}.
\begin{definition}
We consider the following data $(p_*, O, U, J, g, L, h,A)$ satisfying
\begin{enumerate}
\item $(p_*, O, U, J, g)$ is an open K\"{a}hler manifold with a complex structure $J$, a Riemannian metric $g$ and a base point $p_*\in O\subset\subset U$ for an open set $O$.
\item $L\rightarrow U$ is a hermitian line bundle equipped with a hermitian metric $h$ and $A$ is the connection induced by the hermitian metric $h$ on $L$, with its curvature $\Theta(A)=\omega$ which is a K\"{a}hler form of $g$.
\end{enumerate}
The data $(p_*, O, U, J, g, L, h,A)$ is said to satisfy the $H$-condition if there exist a constant $C$ and a compactly supported smooth section $\sigma:U\rightarrow L$ such that
\begin{enumerate}
\item $H_1: ||\sigma||_{L^2(U)}\leq (2\pi)^{\frac{n}{2}}$,
\item $H_2: |\sigma(p_*)|>\frac{3}{4}$,
\item $H_3$: for any smooth section $\tau$ of $L$ over $O$, we have
$$
|\tau(p_*)|\leq C(||\bar{\partial}\tau||_{L^{2n+1}(O)}+||\tau||_{L^2(O)}),
$$
\item $H_4: ||\bar{\partial}\sigma||_{L^2(U)}<\min\big(\frac{1}{8\sqrt{2}C},10^{-20}\big)$,
\item $H_5: ||\bar{\partial}\sigma||_{L^{2n+1}(O)}<\frac{1}{8C}$.
\end{enumerate}
\end{definition}
Fix any point $p$ in $M_1$, $(M_1,p,kd_1)$ converges in pointed Gromov-Hausdorff topology to a tangent cone $C(Y)$ over the cross section when $k\rightarrow \infty$. We still use $p$ for the vertex of $C(Y)$. Let $Y_{reg}$ and $Y_{sing}$ be the regular part and singular part of $Y$ respectively. By \cite{Cheeger}, $Y_{sing}$ has Hausdorff dimension equal or less than $2n-2$. $C(Y_{reg})\backslash \{p\}$ has a natural complex structure induced from the Gromov-Hausdorff limit and the cone metric $g_C$ on $C(Y)$ is given by
$$
\omega_C=\frac{1}{2}\sqrt{-1}\partial\overline{\partial}r^2,
$$
where $r$ is the distance function for any point $z\in C(Y)\backslash p$. According to Proposition \ref{Ana}, the singular set $\mathcal{S}$ of $M_1$ must be a locally analytic set by taking the limit of a divisor on $\overline{M}$. So we also have the following cut-off function on $Y$.
\begin{prop}
For any $\epsilon>0$, there exists a cut-off function $\gamma$ on $Y$ such that
\begin{enumerate}
\item $\gamma\in C^\infty(Y_{reg})$ and $0\leq\gamma\leq1$,
\item $\gamma$ is supported in the $\epsilon$ neighborhood of $Y_{sing}$,
\item $\gamma=1$ on a neighborhood of $Y_{sing}$,
\item $||\nabla \gamma||_{L^2(Y,g_C)}<\epsilon$.
\end{enumerate}
\end{prop}
We consider the trivial line bundle $L_C$ on $C(Y)$ equipped the hermitian metric $h_C=e^{-|z|^2}$, where $|z|^2=r^2$. Then the curvature coincides with $\omega_C$. $A_C$ is denoted by the connection of $L_C$ with metric $h_C$.
\begin{lem}\label{Con}
Let $p_*\in C(Y_{reg})\backslash \{p\}$ such that $\frac{3}{4}<e^{-|p_*|^2}$. Then there exists $U\subset\subset C(Y_{reg})\backslash \{p\}$ and an open neighborhood $O\subset\subset U$ of $p_*$ such that $(p_*,O,U,J_C,g_C,L_C,h_C,A_C)$ satisfies the $H$-condition.
\end{lem}
From the construction in \cite{DS}, $U$ is a product in $C(Y_{reg})\backslash \{p\}$ i.e., there exists $U_{Y}\subset Y_{reg}$ such that $U=\{z=(y,r)\in C(Y)|y\in U_{Y},r\in (r_U,R_U)\}$. For $m\in \mathbb{Z}^+$ defined as \cite{DS} (P79), we define
$$
U(m)=\{z=(y,r)\in C(Y)|y\in U_{Y},r\in (m^{-\frac{1}{2}}r_U,R_U)\}.
$$
For any integer $t$ and $1\leq t\leq m$, $\mu_t:U\rightarrow U(m)$ is defined by $\mu_t(z)=t^{-\frac{1}{2}}z$. The following proposition is due to \cite{DS}.
\begin{prop}\label{Hao}
Suppose $(p_*,O,U(m),J_C,g_C,L_C,h_C,A_C)$ constructed as in Lemma \ref{Con} satisfies the $H$-condition. If $(p_*, O, U(m), J, g, L, h,A)$ satisfies $H$-condition and there exists a small constant $\epsilon>0$ such that
$$
||g-g_C||_{C^0(U(m))}+||J-J_C||_{C^0(U(m))}<\epsilon
$$
Then we can find some $1\leq t\leq m$ such that $(p_*,O, U,\mu^*_tJ,\mu^*_t(tg),\mu^*_t(L^t),\mu^*_t(h^t),\mu^*_t(A^t))$ satisfies $H$-condition.
\end{prop}
Fix any point $p$, we can assume that $(M_1,k_v^{\frac{1}{2}}d_1,p)$ converge to a tangent cone $C(Y_p)$ for some sequence $k_v$ in pointed Gromov-Hausdorff topology. For an open set $U\subset\subset C(Y_{reg})\backslash \{p\}$, there is an embedding $\chi_{k_v}:U\rightarrow \mathcal{R}=M_{reg}$. Note that $d_1|_{\mathcal{R}}=\omega_1$. The following Lemma follows from the convergence of $(M_1,k_v^{\frac{1}{2}}d_1,p)$.
\begin{lem}\label{A}
There exist $v$ and $\epsilon>0$ such that we can find an embedding $\chi_{k_v}$ which satisfies
\begin{enumerate}
\item $\frac{1}{2}|z|\leq k_v^{\frac{1}{2}}d_1(p,\chi_{k_v}(z))\leq 2|z|$,
\item $||\chi_{k_v}^*(k_v\omega_1)-\omega_C||_{C^0(U)}+||\chi_{k_v}^*(J_\mathcal{R})-J_C||_{C^0(U)}<\epsilon$.
\end{enumerate}
\end{lem}
\begin{prop}
For any two distinct point $p$ and $q$ in $M_1$, we have
$$
\Phi_1(p)\neq \Phi_1(q).
$$
\end{prop}
\begin{proof}
Step 1: For any two distinct points $p$ and $q$, there exist $r$ and $R$ such that $p,q\in B_{d_1}(x_1,r)\subset B_{d_1}(x_1,2r)\subset B_{d_1}(x_1,R)$. Suppose $C(Y_p)$ and $C(Y_q)$ are two tangent cones of $p$ and $q$ after rescaling $(M_1,d_1)$ at $p$ by $k_{v_p}\rightarrow \infty$ and at $q$ by $k_{v_q}\rightarrow \infty$. Then according to Lemma \ref{Con}, we can construct two collection of data $(p_*,O_p,U_p(m_p),J_p,g_p,L_p,h_p,A_p)$ and $(q_*,O_q,U_q(m_q),J_q,g_q,L_q,h_q,A_q)$ which satisfy the $H$-condition, where $U_p(m_p)\subset C(Y_p)$ and $U_q(m_q)\subset C(Y_q)$. In addition, we can always assume that
\begin{enumerate}
\item the constant $C$ appeared in the $H$-condition for $U_p(m_p)$ and $U_q(m_q)$ are the same,
\item $k_{v_p}=k_{v_q}=k_{v_{p,q}}$,
\item $r_{p_*}:=d_{C(Y_p)}(p,p_*)$ and $r_{q_*}:=d_{C(Y_q)}(q,q_*)$ are small enough which definite below.
\end{enumerate}

Step 2: From Lemma \ref{A}, there exist $k_{v_{p,q}}$ such that $\chi_{p,k_{v_{p,q}}}: U_p(m_p)\rightarrow \mathcal{R}$ and $\chi_{q,k_{v_{p,q}}}: U_q(m_q)\rightarrow \mathcal{R}$ satisfy the following:
\begin{enumerate}
\item $\frac{1}{2}|z|\leq k_{v_{p,q}}^{\frac{1}{2}}d_1(p,\chi_{p}(z))\leq 2|z|$,
\item $\frac{1}{2}|z|\leq k_{v_{p,q}}^{\frac{1}{2}}d_1(q,\chi_{q}(z))\leq 2|z|$,
\item $\chi_p(U_p(m_p))\cap \chi_q(U_q(m_q))=\varnothing$,
\item $||\chi_{p}^*(k_{v_{p,q}}\omega_1)-\omega_p||_{C^0(U_p(m_p))}+||\chi_{p}^*(J_\mathcal{R})-J_p||_{C^0(U_p(m_p))}<\epsilon$,
\item $||\chi_{q}^*(k_{v_{p,q}}\omega_1)-\omega_q||_{C^0(U_q(m_q))}+||\chi_{q}^*(J_\mathcal{R})-J_q||_{C^0(U_q(m_q))}<\epsilon$.
\end{enumerate}
where for the convenience, $\chi_{p,k_{v_{p,q}}}$ and $\chi_{q,k_{v_{p,q}}}$ are denoted by $\chi_{p}$ and $\chi_{q}$ respectively.

Step 3: By the Proposition \ref{Hao} and sufficiently small $\epsilon$ in Step 2, there exists $1\leq t_p\leq m_p$ such that $(p_*,O_p,U_p,\mu_{t_p}^*\chi_p^*(J_\mathcal{R}),\mu_{t_p}^*\chi_p^*(k_{v_{p,q}}\omega_1),\mu_{t_p}^*\chi_p^*(L_p^{t_p}),\mu_{t_p}^*\chi_p^*(h_{p}^{t_p}),\mu_{t_p}^*\chi_p^*(A_p^{t_p}))$ satisfies the $H$-condition. Thus there is a compactly smooth section $\sigma_p$ such that $\sigma_p$ has properties $H_1$, $H_2$, $H_4$ and $H_5$. By the same argument, there exists $1\leq t_q\leq m_q$ such that
$(q_*,O_q,U_q,\mu_{t_q}^*\chi_q^*(J_\mathcal{R}),\mu_{t_q}^*\chi_q^*(k_{v_{p,q}}\omega_1),\mu_{t_q}^*\chi_q^*(L_q^{t_q}),$ $\mu_{t_q}^*\chi_q^*(h_{q}^{t_q}),\mu_{t_q}^*\chi_q^*(A_q^{t_q}))$ satisfies the $H$-condition. Thus there is a compactly smooth section $\sigma_q$ such that $\sigma_q$ has properties $H_1$, $H_2$, $H_4$ and $H_5$.

Step 4: There is an embedding from $(\mu_{t_p}^*\chi_p^*(L_p^{t_p}),U_p)$ to $(k_p(K_{\overline{M}}+D),\mathcal{R})$, where $k_p=t_pk_{v_{p,q}}$. So $\sigma_p$ can be viewed as a compactly smooth section of $k_p(K_{\overline{M}}+D)$. We now apply Proposition \ref{L2} to $\tau_p=\bar{\partial}\sigma_p$. Then there exists a $(k_p(K_{\overline{M}}+D))$ valued section $\varsigma_p$ solving the $\bar{\partial}$ equation $\bar{\partial}\varsigma_p=\tau_p$ with
$$
\int_{\mathcal{R}}|\varsigma_p|^2(k_p\omega_1)^n\leq \int_{\mathcal{R}}|\tau_p|^2(k_p\omega_1)^n\leq \min\Big(\frac{1}{8\sqrt{2}C},10^{-20}\Big).
$$
Let $z_{p_*}=\chi_p(p_*)$, then from $H_3$ and $H_5$,
$$
\varsigma_p(z_{p_*})\leq C\big(||\bar{\partial}\varsigma_p||_{L^{2n+1}(O_p)}+||\varsigma_p||_{L^2(O_p)}\big)\leq \frac{1}{8\sqrt{2}C}+\frac{1}{8C}\leq \frac{1}{4}.
$$
Set $\sigma'_p=\sigma_p-\varsigma_p$. Then $\sigma'_p$ is a holomorphic section of $k_p(K_{\overline{M}}+D)$ over $\mathcal{R}$ and from Proposition \ref{L0a}, $\sigma'_p$ can be continuously extended to $M_1$. By the $H$-condition, we have the following relations:
\begin{enumerate}
\item $|\sigma'_p(z_{p_*})|>\frac{1}{2}$,
\item $||\sigma'_p||_{L^2(\mathcal{R},k_p\omega_1,h_p^{k_p})}\leq 2(2\pi)^{\frac{n}{2}}$,
\item $||\sigma'_p||_{L^2(\mathcal{R}\backslash U_p,k_p\omega_1,h_p^{k_p})}=||\varsigma_p||_{L^2(\mathcal{R}\backslash U_p,k_p\omega_1,h_p^{k_p})}\leq \min\Big(\frac{1}{8\sqrt{2}C},10^{-20}\Big)$.
\end{enumerate}
Then by Proposition \ref{L0a},
$$
|\sigma'_p(p)|\geq |\sigma'_p(z_{p_*})|-\sup_{B_{d_1}(x_1,r)}|\nabla\sigma'_p|k_p^{\frac{1}{2}}d_1(p,z_{p_*})\geq \frac{2}{5},
$$
when $r_{p_*}$ is sufficiently small.

Now we restrict $\sigma'_p$ on $U_q$. By $H_3$,
$$
|\sigma'_p(z_{q_*})|\leq C||\sigma'_p||_{L^2(\mathcal{R}\backslash U_p,k_p\omega_1,h_p^{k_p})}\leq C\min\Big(\frac{1}{8\sqrt{2}C},10^{-20}\Big).
$$
Similarly,
$$
|\sigma'_p(q)|\leq |\sigma'_p(z_{q_*})|+\sup_{B_{d_1}(x_1,r)}|\nabla\sigma'_p|k_p^{\frac{1}{2}}d_1(q,z_{q_*})\leq 2C\min\Big(\frac{1}{8\sqrt{2}C},10^{-20}\Big),
$$
when $r_{p_*}$ is sufficiently small.

Step 5: By the same argument of Step 4, let $k_q=t_qk_{v_{p,q}}$, we construct a holomorphic section $\sigma'_q$ such that
$$
|\sigma'_q(q)|\geq \frac{2}{5},\ \ \ \ \ |\sigma'_q(p)|\leq 2C\min\Big(\frac{1}{8\sqrt{2}C},10^{-20}\Big).
$$

Step 6: Set $K=t_qk_p=t_pk_q$. Then $(\sigma'_p)^{t_q}$ and $(\sigma'_q)^{t_p}$ are holomorphic section of $K(K_{\overline{M}}+D)$ which can be continuously extended to $M_1$. Modifying the constant $10^{-20}$ as small as enough, we have $|(\sigma'_p)^{t_q}(p)|>>|(\sigma'_q)^{t_p}(p)|$ and $|(\sigma'_q)^{t_p}(q)|>>|(\sigma'_p)^{t_q}(q)|$. Therefore, we conclude that $\Phi_1$ is injective.
\end{proof}
\subsection{Surjectivity of $\Phi_1$}
In this subsection we will complete the proof of Theorem \ref{pr}.
Let $u_1$ be the solution to the following equation in the current sense
$$
(\eta_1-\sqrt{-1}\partial\overline{\partial}\log\log^2|s_D|^2+\sqrt{-1}\partial\overline{\partial}u_1)^n=e^{u_1}\frac{\Omega}{|s_D|^2\log^2|s_D|^2}.
$$
Since $K_{\overline{M}}+D$ is big and semi-ample, there exists an effective divisor $E=\sum_ia_iE_i$ such that $K_{\overline{M}}+D-\epsilon E>0$ for all sufficiently small $\epsilon>0$.

Let $p\in \Supp E\backslash D$ and $\pi: \widetilde{\overline{M}}\rightarrow \overline{M}$ be the blow up at $p$ with exceptional divisor $\pi^{-1}(p)=F$. Set $\widetilde{D}=\pi^{-1}(D)$ and $\widetilde{E}=\sum_ia_i\widetilde{E_i}$, where $\widetilde{E_i}=\overline{\pi^{-1}(E_i)-F}$. $s_{\widetilde{E_i}}$, $s_F$ and $s_{\widetilde{D}}$ are denoted by the defining sections of line bundles $L_{\widetilde{E_i}}$, $L_F$ and $L_{\widetilde{D}}$ respectively. Let $\chi$ be fixed K\"{a}hler metric on $\widetilde{\overline{M}}$. We choose appropriate hermitian metrics $h_{\widetilde{E_i}}$ and $h_F$ such that
$$
\pi^*\eta_1+\delta\sqrt{-1}\partial\overline{\partial}\log|s_F|^2+\delta\sum_ia_i\sqrt{-1}\partial\overline{\partial}\log|s_{\widetilde{E_i}}|^2\geq\mu \chi
$$
for some small constants $\delta$ and $\mu$. Note that $\widetilde{\Omega}=\frac{\pi^*\Omega}{|s_F|^{2(n-1)}}$ defines a smooth volume form on $\widetilde{\overline{M}}$. We consider the following Monge-Amp\`{e}re equation on $\widetilde{\overline{M}}$
\begin{equation}\label{use}
(\widetilde{\eta_1}+\epsilon\chi+\sqrt{-1}\partial\overline{\partial}\varphi_\epsilon)^n=e^{\varphi_\epsilon}(\epsilon^2+|s_F|^2)^{n-1}\frac{\widetilde{\Omega}}{|s_{\widetilde{D}}|^2_{h_{\widetilde{D}}}\log^2|s_{\widetilde{D}}|^2_{h_{\widetilde{D}}}}
\end{equation}
where $\widetilde{\eta_1}=\pi^*\eta_1-\sqrt{-1}\partial\overline{\partial}\log\log^2|s_{\widetilde{D}}|^2_{h_{\widetilde{D}}}$ and $h_{\widetilde{D}}=\pi^*h_D$. By Theorem 1 of \cite{Ko84}, the equation has a unique smooth solution $\varphi_\epsilon$ for each $\epsilon$; moreover
$$
\widetilde{\omega_\epsilon}:=\widetilde{\eta_1}+\epsilon\chi+\sqrt{-1}\partial\overline{\partial}\varphi_\epsilon
$$
is a smooth complete K\"{a}hler metric on $\widetilde{\overline{M}}\backslash \widetilde{D}$.
\begin{lem}\label{c0}
For any $\delta$ and $\epsilon$, there exist two constants $C(\delta)$ and $C$ independent of $\epsilon$ such that
$$
-C(\delta)+\delta\log|s_F|^2+\delta\sum_ia_i\log|s_{\widetilde{E_i}}|^2\leq \varphi_\epsilon\leq C+\log\log^2|s_{\widetilde{D}}|^2_{h_{\widetilde{D}}}.
$$
\end{lem}
\begin{proof}
For the upper bound, let
$$
V_\epsilon=\int(\epsilon^2+|s_F|^2)^{n-1}\frac{\widetilde{\Omega}}{|s_{\widetilde{D}}|^2_{h_{\widetilde{D}}}\log^2|s_{\widetilde{D}}|^2_{h_{\widetilde{D}}}},
$$
so we have $V_1\geq V_\epsilon\geq V_0$. Hence $V_\epsilon$ is uniformly bounded. Denote $(\epsilon^2+|s_F|^2)^{n-1}\frac{\widetilde{\Omega}}{|s_{\widetilde{D}}|^2_{h_{\widetilde{D}}}\log^2|s_{\widetilde{D}}|^2_{h_{\widetilde{D}}}}$ by $\widetilde{\Omega_\epsilon}$, then we have the following calculation
\begin{align*}
\frac{1}{V_\epsilon}\int\varphi_\epsilon\widetilde{\Omega_\epsilon} &=\frac{1}{V_\epsilon}\int\log\Bigg(\frac{\widetilde{\omega_\epsilon}^n}{\widetilde{\Omega_\epsilon}}\Bigg)\widetilde{\Omega_\epsilon} \leq \log\int\widetilde{\omega_\epsilon}^n-\log V_\epsilon \\
&=\log \int(\pi^*\eta+\epsilon\chi)^n-\log V_\epsilon\leq C,
\end{align*}
where the third equality bases on a Lemma (\cite{Ko84} P410). Since $\varphi_\epsilon-\log\log^2|s_{\widetilde{D}}|^2_{h_{\widetilde{D}}} \in PSH(\widetilde{\overline{M}},\pi^*\eta+\epsilon\chi)$, the mean inequality implies
$$
\sup\varphi_\epsilon\leq C+\log\log^2|s_{\widetilde{D}}|^2_{h_{\widetilde{D}}}.
$$

For the lower bound, we set $\varphi_{\epsilon,\delta}=\varphi_\epsilon-\delta\log|s_F|^2-\delta\sum_ia_i\log|s_{\widetilde{E_i}}|^2$ and denote $|s_{\widetilde{D}}|^2_{h_{\widetilde{D}}^\delta}=|s_{\widetilde{D}}|^2_{\delta}$, then the equation (\ref{use}) is equivalent to
\begin{align*}
& \Big(\eta^\delta_1+\delta\sqrt{-1}\partial\overline{\partial}\log|s_F|^2+\delta\sum_ia_i\sqrt{-1}\partial\overline{\partial}\log|s_{\widetilde{E_i}}|^2
+\epsilon\chi+\sqrt{-1}\partial\overline{\partial}\varphi_{\epsilon,\delta}+\sqrt{-1}\partial\overline{\partial}\log\frac{\log^2|s_{\widetilde{D}}|^2_{\delta}}{\log^2|s_{\widetilde{D}}|^2}\Big)^n\\
& = e^{\varphi_{\epsilon,\delta}+\log\frac{\log^2|s_{\widetilde{D}}|^2_{\delta}}{\log^2|s_{\widetilde{D}}|^2}}\cdot\prod_i|s_{\widetilde{E_i}}|^{2a_i\delta}\cdot|s_F|^{2\delta}
\cdot(\epsilon^2+|s_F|^2)^{n-1}\cdot\frac{\widetilde{\Omega}'}{|s_{\widetilde{D}}|^2_{\delta}\log^2|s_{\widetilde{D}}|^2_{\delta}}
\end{align*}
where $\eta^\delta_1=\pi^*\eta_1-\sqrt{-1}\partial\overline{\partial}\log\log^2|s_{\widetilde{D}}|^2_\delta$ satisfying $\eta^\delta_1+\delta\sqrt{-1}\partial\overline{\partial}\log|s_F|^2+\delta\sum_ia_i\sqrt{-1}\partial\overline{\partial}\log|s_{\widetilde{E_i}}|^2>0$ and $\widetilde{\Omega}'=\frac{|s_{\widetilde{D}}|^2_\delta}{|s_{\widetilde{D}}|^2}\widetilde{\Omega}$.
We introduce the following equation
$$
(\eta^\delta_1-\delta\sum_i\Theta_{\widetilde{E_i}}-\delta\Theta_F+\epsilon\chi+\sqrt{-1}\partial\overline{\partial}\psi_{\epsilon,\delta})^n=e^{\psi_{\epsilon,\delta}}\cdot(\epsilon^2+|s_F|^2)
^{n-1}\cdot\frac{\widetilde{\Omega}'}{|s_{\widetilde{D}}|^2_{\delta}\log^2|s_{\widetilde{D}}|^2_{\delta}}.
$$
By the generalized maximum principle, there exists a sequence $\{x_i\}$ such that $\lim_{i\rightarrow\infty}\psi_{\epsilon,\delta}(x_i)=\inf\psi_{\epsilon,\delta}$ and $\overline{\lim}_{i\rightarrow\infty}\sqrt{-1}\partial\overline{\partial}\psi_{\epsilon,\delta}(x_i)\geq 0$. Then we have
$$
\inf\psi_{\epsilon,\delta}\geq (n-1)\log\frac{1}{\epsilon^2+|s_F|^2}+\log\frac{|s_{\widetilde{D}}|^2_{\delta}\log^2|s_{\widetilde{D}}|^2_{\delta}(\eta^\delta_1-\delta\sum_i\Theta_{\widetilde{E_i}}-\delta\Theta_F+\epsilon\chi)^n}{\widetilde{\Omega}'}\geq-C(\delta).
$$
Set $H_{\epsilon,\delta}=\varphi_{\epsilon,\delta}-\psi_{\epsilon,\delta}$ and $v_\epsilon^\delta=\eta^\delta_1-\delta\sum_i\Theta_{\widetilde{E_i}}-\delta\Theta_F+\epsilon\chi$, then
\begin{align*}
&\log\frac{\Big(v_\epsilon^\delta+\sqrt{-1}\partial\overline{\partial}\psi_{\epsilon,\delta}+\sqrt{-1}\partial\overline{\partial}H_{\epsilon,\delta}+\sqrt{-1}\partial\overline{\partial}\log\frac{\log^2|s_{\widetilde{D}}|^2_{\delta}}{\log^2|s_{\widetilde{D}}|^2}\Big)^n}
{(v_\epsilon^\delta+\sqrt{-1}\partial\overline{\partial}\psi_{\epsilon,\delta})^n}\\
=& H_{\epsilon,\delta}+\log\frac{\log^2|s_{\widetilde{D}}|^2_{\delta}}{\log^2|s_{\widetilde{D}}|^2}+
\delta\log|s_F|^2+\delta\sum_ia_i\log|s_{\widetilde{E_i}}|^2.
\end{align*}
By the generalized maximum principle again
$$
\inf\Bigg(H_{\epsilon,\delta}+\log\frac{\log^2|s_{\widetilde{D}}|^2_{\delta}}{\log^2|s_{\widetilde{D}}|^2}\Bigg)\geq -C(\delta).
$$
Note that $\log\frac{\log^2|s_{\widetilde{D}}|^2_{\delta}}{\log^2|s_{\widetilde{D}}|^2}$ is a smooth function on $\widetilde{\overline{M}}$, so it can be bounded by $C(\delta)$. Moreover we get the lower bound of $\varphi_{\epsilon}$.
\end{proof}
\begin{lem}\label{Ric}
There exists a constant $C$ independent of $\epsilon$ such that on $\widetilde{\overline{M}}\backslash \widetilde{D}$, we have
$$
\Ric(\widetilde{\omega_{\epsilon}})\leq -\widetilde{\omega_{\epsilon}}+C\chi.
$$
\end{lem}
\begin{proof}
We observe some following consequences:
\begin{enumerate}
\item $\pi^*\eta_1\leq C\chi$,
\item Since $\widetilde{\Omega}$ is a smooth volume form, $\Ric(\widetilde{\Omega})\leq C\chi$,
\item $\Theta_{h_{\widetilde{D}}}\geq -C\chi$,
\item $\sqrt{-1}\partial\overline{\partial}\log (\epsilon^2+|s_F|^2)\geq -C\chi$.
\end{enumerate}
Thus by a simple calculation we get the Lemma.
\end{proof}
Set $\chi'=\chi-\sqrt{-1}\partial\overline{\partial}\log\log^2|s_{\widetilde{D}}|^2_{'}$, then by a calculation we have
$$
\chi'=\chi-2\frac{\sqrt{-1}\partial\overline{\partial}\log|s_{\widetilde{D}}|^2_{'}}{\log|s_{\widetilde{D}}|^2_{'}}
+2\frac{\sqrt{-1}\partial\log|s_{\widetilde{D}}|^2_{'}\wedge\overline{\partial}\log|s_{\widetilde{D}}|^2_{'}}{\log^2|s_{\widetilde{D}}|^2_{'}}.
$$
Take an appropriate hermitian metric $|\cdot|_{'}$, we can assume that
$$
\frac{1}{2}\chi\leq \frac{1}{2}\chi+2\frac{\sqrt{-1}\partial\log|s_{\widetilde{D}}|^2_{'}\wedge\overline{\partial}\log|s_{\widetilde{D}}|^2_{'}}{\log^2|s_{\widetilde{D}}|^2_{'}}\leq\chi'\leq 2\chi+2\frac{\sqrt{-1}\partial\log|s_{\widetilde{D}}|^2_{'}\wedge\overline{\partial}\log|s_{\widetilde{D}}|^2_{'}}{\log^2|s_{\widetilde{D}}|^2_{'}}.
$$
So, by Lemma \ref{Ric}, we have
$$
\Ric(\widetilde{\omega_{\epsilon}})\leq -\widetilde{\omega_{\epsilon}}+C\chi'.
$$
On the other hand, we can choose a sufficiently large $A'$, a sufficiently small $\alpha$ and a hermitian metric $|\cdot|$ such that
\begin{align*}
&A'\pi^*\eta_1-\Theta_F-\sum_ia_i\Theta_{\widetilde{E_i}}-A'\sqrt{-1}\partial\overline{\partial}\log\log^2|s_{\widetilde{D}}|^2\\
\geq &
3\chi-2A'\frac{\sqrt{-1}\partial\overline{\partial}\log|s_{\widetilde{D}}|^2}{\log|s_{\widetilde{D}}|^2}+2A'\frac{\sqrt{-1}\partial\log|s_{\widetilde{D}}|^2\wedge\overline{\partial}\log|s_{\widetilde{D}}|^2}{\log^2|s_{\widetilde{D}}|^2}\\
\geq & 2\chi+2A'\frac{\sqrt{-1}\partial\log|s_{\widetilde{D}}|^2\wedge\overline{\partial}\log|s_{\widetilde{D}}|^2}{\log^2|s_{\widetilde{D}}|^2}\geq \alpha \chi'.
\end{align*}
From now on we always assume that the hermitian metric $|\cdot|$ on $L_{\widetilde{D}}$ satisfy $A'\pi^*\eta_1-\Theta_F-\sum_ia_i\Theta_{\widetilde{E_i}}-A'\sqrt{-1}\partial\overline{\partial}\log\log^2|s_{\widetilde{D}}|^2\geq \alpha \chi'$.
\begin{lem}\label{C2}
There exist $C$ and $\lambda$ independent of $\epsilon$ such that
$$
\widetilde{\omega_{\epsilon}}\leq \frac{C\big(\log^2|s_{\widetilde{D}}|^2\big)^C}{|s_{\widetilde{D}}|^{2\lambda}\cdot|s_F|^{2\lambda^2}\cdot\prod_i|s_{\widetilde{E_i}}|^{2\lambda^2}}\chi'.
$$
\end{lem}
\begin{proof}
By Yau's Schwarz Lemma \cite{Y} and Lemma \ref{Ric}, we have
$$
\triangle_{\widetilde{\omega_{\epsilon}}}\log tr_{\chi'}\widetilde{\omega_{\epsilon}}\geq-Ctr_{\widetilde{\omega_{\epsilon}}}\chi'-\frac{C}{tr_{\chi'}\widetilde{\omega_{\epsilon}}}.
$$
There is a fact that is
$$
\triangle_{\widetilde{\omega_{\epsilon}}}\varphi_\epsilon\leq n-tr_{\widetilde{\omega_{\epsilon}}}(\pi^*\eta_1-\sqrt{-1}\partial\overline{\partial}\log\log^2|s_{\widetilde{D}}|^2).
$$
Let $H=\log(|s_{\widetilde{D}}|^{2A}\cdot|s_F|^{2A^2}\cdot\prod_i|s_{\widetilde{E_i}}|^{2A^2}\cdot tr_{\chi'}\widetilde{\omega_{\epsilon}})-A^2A'\varphi_\epsilon$, where $A'$ is chosen as above and $A$ is defined below. Then on $\widetilde{\overline{M}}\backslash (\widetilde{D}\cup F \cup \Supp E)$, we have
$$
\triangle_{\widetilde{\omega_{\epsilon}}}H\geq -Ctr_{\widetilde{\omega_{\epsilon}}}\chi'-\frac{C}{tr_{\chi'}\widetilde{\omega_{\epsilon}}}-A^2A'n+Atr_{\widetilde{\omega_{\epsilon}}}
(AA'\pi^*\eta_1-A\Theta_F-A\sum_ia_i\Theta_{\widetilde{E_i}}-\Theta_{\widetilde{D}}-AA'\sqrt{-1}\partial\overline{\partial}\log\log^2|s_{\widetilde{D}}|^2).
$$
When $A$ is sufficiently large we observed that
$$
Atr_{\widetilde{\omega_{\epsilon}}}(A(A'\pi^*\eta_1-\Theta_F-\sum_ia_i\Theta_{\widetilde{E_i}}-A'\sqrt{-1}\partial\overline{\partial}\log\log^2|s_{\widetilde{D}}|^2)-\Theta_{\widetilde{D}})
\geq (C+1)\chi'.
$$
Therefore
$$
\triangle_{\widetilde{\omega_{\epsilon}}}H\geq tr_{\widetilde{\omega_{\epsilon}}}\chi'-\frac{C}{tr_{\chi'}\widetilde{\omega_{\epsilon}}}-A^2A'n.
$$
By the generalized maximum principle, there exists a sequence $\{x_i\}$ such that $\lim_{i\rightarrow\infty} H(x_i)=\sup H$ and $\overline{\lim}_{i\rightarrow\infty}\sqrt{-1}\partial\overline{\partial}H(x_i)\leq 0$. Thus,
$$
\overline{\lim}_{i\rightarrow\infty}tr_{\chi'}\widetilde{\omega_{\epsilon}}\cdot(tr_{\widetilde{\omega_{\epsilon}}}\chi'-A^2A'n)(x_i)\leq C.
$$
Since
$$
\widetilde{\omega_\epsilon}^n=e^{\varphi_\epsilon}\frac{\widetilde{\Omega}}{|s_{\widetilde{D}}|^2\log^2|s_{\widetilde{D}}|^2}\leq C\log^2|s_{\widetilde{D}}|^2(\chi')^n,
$$
then we have
$$
\frac{1}{\log^2|s_{\widetilde{D}}|^2}(tr_{\chi'}\widetilde{\omega_\epsilon})^{\frac{1}{n-1}}\leq Ctr_{\widetilde{\omega_\epsilon}}\chi'.
$$
Furthermore,
\begin{equation}\label{euse}
\overline{\lim}_{i\rightarrow\infty}tr_{\chi'}\widetilde{\omega_{\epsilon}}\cdot \bigg(\frac{1}{C\log^2|s_{\widetilde{D}}|^2}(tr_{\chi'}\widetilde{\omega_\epsilon})^{\frac{1}{n-1}}-A^2A'n\bigg)(x_i)\leq C.
\end{equation}
If
$$
\overline{\lim}_{i\rightarrow\infty}(tr_{\chi'}\widetilde{\omega_\epsilon})^{\frac{1}{n-1}}(x_i)\leq \overline{\lim}_{i\rightarrow\infty}2A^2A'nC\log^2|s_{\widetilde{D}}|^2(x_i),
$$
then
$$
\overline{\lim}_{i\rightarrow\infty}(tr_{\chi'}\widetilde{\omega_\epsilon})(x_i)\leq \overline{\lim}_{i\rightarrow\infty}(2A^2A'nC)^{n-1}(\log^2|s_{\widetilde{D}}|^2)^{n-1}(x_i).
$$
Otherwise
$$
\overline{\lim}_{i\rightarrow\infty}(tr_{\chi'}\widetilde{\omega_\epsilon})^{\frac{1}{n-1}}(x_i)\geq \overline{\lim}_{i\rightarrow\infty}2A^2A'nC\log^2|s_{\widetilde{D}}|^2(x_i).
$$
From (\ref{euse}) we know
$$
\overline{\lim}_{i\rightarrow\infty}A^2A'ntr_{\chi'}\widetilde{\omega_\epsilon}(x_i)\leq C.
$$
In general we have
$$
\overline{\lim}_{i\rightarrow\infty}(tr_{\chi'}\widetilde{\omega_\epsilon})(x_i)\leq \overline{\lim}_{i\rightarrow\infty}C(\log^2|s_{\widetilde{D}}|^2)^{C}(x_i).
$$
By the definition of $H$ and Lemma \ref{c0} we have
\begin{align*}
H(x) & \leq \overline{\lim}_{i\rightarrow\infty}\bigg(\log\big(|s_{\widetilde{D}}|^{2A}\cdot|s_F|^{2A^2}\cdot\prod_i|s_{\widetilde{E_i}}|^{2A^2}\cdot C(\log^2|s_{\widetilde{D}}|^2)^C\big)\\
& \qquad +A^2A'C(\delta)-A^2A'\delta\log|s_F|^2-A^2A'\delta\sum_ia_i\log|s_{\widetilde{E_i}}|^2\bigg)(x_i)\leq C
\end{align*}
when choosing $A>> A'$ and sufficiently small $\delta$. So we get this Lemma from the upper bound of $\varphi_\epsilon$.
\end{proof}
Let $B$ be a disk in $\overline{M}\backslash D$ centered at $p$. Denote $f_1,f_2,\cdot\cdot\cdot, f_N$ by the defining functions of divisors $\widetilde{E_1},\widetilde{E_2},\cdot\cdot\cdot,\widetilde{E_N}$ on $\widetilde{B}=\pi^{-1}(B)$. From Lemma \ref{C2}, we obtain the following corollary.
\begin{cor}\label{Cor}
There exist $C$ and $\lambda$ independent of $\epsilon$ such that
$$
(tr_{\chi'}\widetilde{\omega_\epsilon})|_{\partial\widetilde{B}}\leq \frac{C}{\prod_i|f_i|^{2\lambda^2}}\bigg|_{\partial\widetilde{B}}.
$$
\end{cor}
Let $\hat{\chi}$ be the pull back of the Euclidean metric $\sqrt{-1}\sum_jdz_j\wedge d\bar{z_j}$ on $B$. Then $\hat{\chi}$ is a smooth closed nonnegative $(1,1)$ form and is a K\"{a}hler metric on $\widetilde{B}\backslash F$. The following Lemma is due to Song \cite{So}.
\begin{lem}\label{L533}
There exist a constant $C>0$, a sufficiently small $\epsilon_0>0$ and a smooth hermitian metric $h'_F$ on $L_F$ such that on $\widetilde{B}$
$$
C^{-1}\hat{\chi}\leq \chi'\leq C\frac{\hat{\chi}}{|s_F|^2_{h'_F}}
$$
and
$$
\pi^*\eta_1-\sqrt{-1}\partial\overline{\partial}\log\log^2|s_{\widetilde{D}}|^2-\epsilon_0\Theta_{h'_F}>C^{-1}\chi'.
$$
\end{lem}
\begin{lem}\label{distance}
There exist $0<\beta<1$, $C>0$ and $\Lambda>0$ independent of $\epsilon$ such that
$$
\widetilde{\omega_\epsilon}\leq \frac{C}{|s_F|^{2(1-\beta)}_{h'_F}\cdot \prod_i|f_i|^{2\Lambda}}\chi', \ \ in \ \ \widetilde{B}.
$$
Moreover, we have
$$
\pi^*\omega_1\leq \frac{C}{|s_F|^{2(1-\beta)}_{h'_F}\cdot \prod_i|f_i|^{2\Lambda}}\chi', \ \ in \ \ \widetilde{B},
$$
where $\omega_1=\eta_1-\sqrt{-1}\partial\overline{\partial}\log\log^2|s_D|^2+\sqrt{-1}\partial\overline{\partial}u_1$.
\end{lem}
\begin{proof}
Let $H=\log\Big(|s_F|^{2(1+r)}_{h'_F}\cdot\prod_i|f_i|^{2\lambda^2}\cdot tr_{\hat{\chi}}\widetilde{\omega_\epsilon}\Big)-A\varphi_\epsilon$ for some sufficiently large $A$ and sufficiently small $r$. There are some facts on $\widetilde{B}\backslash (F\cup \Supp E)$:
\begin{enumerate}
\item $\triangle_{\widetilde{\omega_\epsilon}}\log|s_F|^2_{h'_F}=-tr_{\widetilde{\omega_\epsilon}}\Theta_{h'_F}$,
\item $\triangle_{\widetilde{\omega_\epsilon}}\log \prod_i|f_i|^{2\lambda^2}=0$,
\item $\triangle_{\widetilde{\omega_\epsilon}}\varphi_\epsilon=n-tr_{\widetilde{\omega_{\epsilon}}}(\pi^*\eta_1-\sqrt{-1}\partial\overline{\partial}\log\log^2|s_{\widetilde{D}}|^2)-\epsilon tr_{\widetilde{\omega_\epsilon}}\chi$,
\item $\triangle_{\widetilde{\omega_\epsilon}}\log tr_{\hat{\chi}}\widetilde{\omega_\epsilon}\geq -\frac{tr_{\hat{\chi}}(\Ric(\widetilde{\omega_\epsilon}))}{tr_{\hat{\chi}}\widetilde{\omega_\epsilon}}\geq 1-\frac{C}{|s_F|^2_{h'_F}tr_{\hat{\chi}}\widetilde{\omega_\epsilon}}$.
\end{enumerate}
Thus, on $\widetilde{B}\backslash (F\cup \Supp E)$, we have
\begin{align*}
 \triangle_{\widetilde{\omega_\epsilon}}H & \geq 1-\frac{C}{|s_F|^2_{h'_F}tr_{\hat{\chi}}\widetilde{\omega_\epsilon}}-An-(r+1)tr_{\widetilde{\omega_\epsilon}}\Theta_{h'_F}+Atr_{\widetilde{\omega_{\epsilon}}}(\pi^*\eta_1-\sqrt{-1}\partial\overline{\partial}\log\log^2|s_{\widetilde{D}}|^2)\\
 & \geq c tr_{\widetilde{\omega_\epsilon}}\chi'-\frac{C}{|s_F|^2_{h'_F}tr_{\hat{\chi}}\widetilde{\omega_\epsilon}}-C,
\end{align*}
where the last inequality bases on Lemma \ref{L533} and choosing sufficiently large $A$ and small $r$. By Yau's Schwarz Lemma \cite{Y},
$$
 \triangle_{\widetilde{\omega_\epsilon}}\log tr_{\chi'}\widetilde{\omega_\epsilon}\geq -C_1tr_{\widetilde{\omega_\epsilon}}\chi'-\frac{C_1}{tr_{\chi'}\widetilde{\omega_\epsilon}}.
$$
Let $G=H+\frac{c}{2C_1}\log\big(\prod_i|f_i|^{2\lambda^2+2}\cdot tr_{\chi'}\widetilde{\omega_\epsilon}\big)$. Then
\begin{align*}
\triangle_{\widetilde{\omega_\epsilon}}G & \geq c tr_{\widetilde{\omega_\epsilon}}\chi'-\frac{C}{|s_F|^2_{h'_F}tr_{\hat{\chi}}\widetilde{\omega_\epsilon}}-C-\frac{c}{2}tr_{\widetilde{\omega_\epsilon}}\chi'-\frac{c}{2tr_{\chi'}\widetilde{\omega_\epsilon}}\\
& \geq \frac{c}{2}tr_{\widetilde{\omega_\epsilon}}\chi'-\bar{C}-\frac{C}{|s_F|^2_{h'_F}tr_{\hat{\chi}}\widetilde{\omega_\epsilon}}.
\end{align*}
For a fixed sufficiently large $\lambda>0$, there exists a constant $C>0$ such that
$$
\sup_{\partial\widetilde{B}}G\leq C
$$
from the estimate in Corollary \ref{Cor} and Lemma \ref{c0}.

So we can assume that
$$
\sup_{\widetilde{B}}G=G(p_{max})
$$
for some $p_{max}\in \widetilde{B}\backslash (F\cup \Supp E)$. Then at $p_{max}$, we have
$$
|s_F|^2_{h'_F}tr_{\hat{\chi}}\widetilde{\omega_\epsilon}(c tr_{\widetilde{\omega_\epsilon}}\chi'-2\bar{C})(p_{max})\leq C.
$$
Note that
$$
\frac{1}{\log^2|s_{\widetilde{D}}|^2}(tr_{\chi'}\widetilde{\omega_\epsilon})^{\frac{1}{n-1}}\leq Ctr_{\widetilde{\omega_\epsilon}}\chi'.
$$
Then according to the boundedness of $\frac{1}{\log^2|s_{\widetilde{D}}|^2}$ in $\widetilde{B}$, we get
\begin{equation}\label{ee5}
|s_F|^2_{h'_F}tr_{\hat{\chi}}\widetilde{\omega_\epsilon}\big(C'(tr_{\chi'}\widetilde{\omega_\epsilon})^{\frac{1}{n-1}}-2\bar{C}\big)(p_{max})\leq C.
\end{equation}
If $(tr_{\chi'}\widetilde{\omega_\epsilon})^{\frac{1}{n-1}}(p_{max})\leq \frac{3\bar{C}}{C'}$, then $G$ is bounded from above by a uniform constant.

Otherwise $(tr_{\chi'}\widetilde{\omega_\epsilon})^{\frac{1}{n-1}}(p_{max})\geq\frac{3\bar{C}}{C'}$, i.e., $\bar{C}\leq \frac{C'}{3}(tr_{\chi'}\widetilde{\omega_\epsilon})^{\frac{1}{n-1}}(p_{max})$. Then by equation (\ref{ee5}) we get
$$
|s_F|^2_{h'_F}\cdot tr_{\hat{\chi}}\widetilde{\omega_\epsilon}\cdot \frac{C'}{3}(tr_{\chi'}\widetilde{\omega_\epsilon})^{\frac{1}{n-1}}(p_{max})\leq C,
$$
i.e.
$$
\log|s_F|^2_{h'_F}+\log tr_{\hat{\chi}}\widetilde{\omega_\epsilon}+\frac{1}{n-1}\log tr_{\chi'}\widetilde{\omega_\epsilon}(p_{max})\leq C.
$$
According to the definition of $G$, Lemma \ref{c0} and Lemma \ref{C2} we have $G\leq C$ when we choose large $C_1$.

In sum, in all cases, we have $G\leq C$. Then
$$
|s_F|^{2(1+r)}_{h'_F}\cdot\prod_i|f_i|^{2\lambda^2+\frac{c}{2C_1}(2\lambda^2+2)}\cdot tr_{\hat{\chi}}\widetilde{\omega_\epsilon}\cdot (tr_{\chi'}\widetilde{\omega_\epsilon})^{\frac{c}{2C_1}}\leq C.
$$
Note that $tr_{\hat{\chi}}\widetilde{\omega_{\epsilon}}\geq C^{-1}tr_{\chi'}\widetilde{\omega_{\epsilon}}$, then we observe
$$
(tr_{\chi'}\widetilde{\omega_\epsilon})^{1+\frac{c}{2C_1}}\leq\frac{C}{|s_F|^{2(1+r)}_{h'_F}\cdot\prod_i|f_i|^{2\lambda^2+\frac{c}{2C_1}(2\lambda^2+2)}}.
$$
If we choose $r=\frac{c}{10C_1}$, then $1-\beta=\frac{1+r}{1+\frac{c}{2c_1}}$ for some $\beta\in (0,1)$. Furthermore, there exists a constant $\Lambda>0$ such that
$$
\widetilde{\omega_\epsilon}\leq \frac{C}{|s_F|^{2(1-\beta)}_{h'_F}\cdot \prod_i|f_i|^{2\Lambda}}\chi'.
$$
\end{proof}

From now on we turn to the Gromov-Hausdorff convergence. Recall
$$
\overline{ D_{1}}:=\{x\in M_1|there \ exists \ x_i\in \overline{D} such \ that \ x_i\xrightarrow{d_{GH}} x\},
$$
where $\overline{D}$ is a divisor such that $D\cup \mathcal{S}_{\overline{M}}\subset \overline{D}$. By the Proposition \ref{Ana}, $(M_1\backslash \overline{ D_{1}},\overline{\omega_1})$ isometric to $(\overline{M}\backslash \overline{D},\omega_1)$.
\begin{lem}
$\Phi_{1}: M_1\setminus \overline{D_1}\rightarrow \Phi(\overline{M}\setminus \overline{D})$ is bijective.
\end{lem}
\begin{proof}
Note that $(M_1\setminus \overline{D_1})\subset \mathcal{R}=M_{reg}$ and $\Phi|_{M_{reg}}$ is biholomorphic, so $\Phi_1$ is bijective.
\end{proof}
\begin{lem}
$\Phi_{1}:\overline{D_1}\rightarrow \Phi(\overline{D}\backslash D)$ is surjective.
\end{lem}
\begin{proof}
For any $x'\in \overline{D}\backslash D$, there exists a curve $\gamma: [0, 1]\rightarrow \overline{M}\backslash D$ with $\gamma(0)=x'$ and $\gamma((0, 1])\subset \overline{M}\backslash \overline{D}$ such that $\int_0^1|\dot{\gamma}|_{\omega_1}dt<\infty$ by Lemma \ref{distance}. The curve $\gamma(t)$ gives a curve $\bar{\gamma}(t)$ for $0<t\leq 1$ through an isometry from
$(\overline{M}\backslash \overline{D},\omega_1)$ to $(M_1\backslash \overline{ D_{1}},\overline{\omega_1})$. Hence there is a limit $x''=\lim_{t\rightarrow 0}\bar{\gamma}(t)$ in $M_1$. Then
$$
\Phi_1(x'')=\lim_{t\rightarrow 0}\Phi_1(\bar{\gamma}(t))=\lim_{t\rightarrow 0}\Phi(\gamma(t))=\Phi(x').
$$
Therefore, $\Phi_{1}$ is surjective.
\end{proof}

\end{document}